\documentclass[12pt,american]{article}
\usepackage[T1]{fontenc}
\usepackage[latin9]{inputenc}
\usepackage{xcolor}
\usepackage{pdfcolmk}
\usepackage{babel}
\usepackage{varioref}
\usepackage{mathrsfs}
\usepackage{amsmath}
\usepackage{amsthm}
\usepackage{amssymb}
\PassOptionsToPackage{normalem}{ulem}
\usepackage{ulem}
\usepackage[unicode=true,
 bookmarks=true,bookmarksnumbered=false,bookmarksopen=false,
 breaklinks=true,pdfborder={0 0 0},pdfborderstyle={},backref=false,colorlinks=true]
 {hyperref}
\hypersetup{pdftitle={Appell System for one Dimensional Fractional Poisson Measure},
 pdfauthor={Jerome B. Bendong and Sheila M. Menchavez  and Jose Luis da Silva},
 pdfsubject={Non Gaussian Measures},
 pdfkeywords={Fractional Poisson measure,  Appell polynomials},
 linkcolor=black, citecolor=blue, urlcolor=blue, filecolor=red}

\makeatletter

\pdfpageheight\paperheight
\pdfpagewidth\paperwidth

\providecolor{lyxadded}{rgb}{0,0,1}
\providecolor{lyxdeleted}{rgb}{1,0,0}
\DeclareRobustCommand{\lyxsout}[1]{\ifx\\#1\else\sout{#1}\fi}

\theoremstyle{plain}
\newtheorem{thm}{\protect\theoremname}[section]
\theoremstyle{remark}
\newtheorem{rem}[thm]{\protect\remarkname}
\theoremstyle{plain}
\newtheorem{lem}[thm]{\protect\lemmaname}
\theoremstyle{plain}
\newtheorem{prop}[thm]{\protect\propositionname}
\theoremstyle{plain}
\newtheorem{cor}[thm]{\protect\corollaryname}
\theoremstyle{definition}
\newtheorem{example}[thm]{\protect\examplename}
\theoremstyle{definition}
\newtheorem{defn}[thm]{\protect\definitionname}

\usepackage{babel}
\usepackage{xcolor}
\usepackage{pdfcolmk}
\usepackage{bm}%\usepackage[]{ulem}

\usepackage{datetime}
\usepackage{currfile}
\usepackage{pdfsync}
\usepackage{bbm,geometry}

\numberwithin{equation}{section}

\providecommand{\corollaryname}{Corollary}
\providecommand{\definitionname}{Definition}
\providecommand{\examplename}{Example}
\providecommand{\remarkname}{Remark}
\providecommand{\theoremname}{Theorem}

\newcommand{\vertiii}[1]{{\left\vert\kern-0.4ex\left\vert\kern-0.4ex\left\vert #1 
    \right\vert\kern-.4ex\right\vert\kern-0.4ex\right\vert}}

\providecommand{\propositionname}{Proposition}

\makeatother

\providecommand{\corollaryname}{Corollary}
\providecommand{\definitionname}{Definition}
\providecommand{\examplename}{Example}
\providecommand{\lemmaname}{Lemma}
\providecommand{\propositionname}{Proposition}
\providecommand{\remarkname}{Remark}
\providecommand{\theoremname}{Theorem}

\begin{document}
\title{Fractional Poisson Analysis in Dimension one}
\author{\textbf{Jerome B. Bendong}\\
 CSM, MSU-Iligan Institute of Technology,\\
 Tibanga, 9200 Iligan City, Philippines\\
 Email: jerome.bendong@g.msuiit.edu.ph\and \textbf{Sheila M. Menchavez}\\
 CSM, MSU-Iligan Institute of Technology,\\
 Tibanga, 9200 Iligan City, Philippines\\
 Email: sheila.menchavez@g.msuiit.edu.ph\and \textbf{Jos{\'e} Lu{\'\i}s
da Silva}\\
Faculdade de Ci{\^e}ncias Exatas e da Engenharia,\\
CIMA, Universidade da Madeira,\\
Campus Universit{\'a}rio da Penteada,\\
9020-105 Funchal, Portugal\\
 Email: joses@staff.uma.pt}
\date{\today}
\maketitle
\begin{abstract}
In this paper, we use a biorthogonal approach (Appell system) to construct
and characterize the spaces of test and generalized functions associated
to the fractional Poisson measure $\pi_{\lambda,\beta}$, that is,
a probability measure in the set of natural (or real) numbers. The
Hilbert space $L^{2}(\pi_{\lambda,\beta})$ of complex-valued functions
plays a central role in the construction, namely, the test function
spaces $(N)_{\pi_{\lambda,\beta}}^{\kappa}$, $\kappa\in[0,1]$ is
densely embedded in $L^{2}(\pi_{\lambda,\beta})$. Moreover, $L^{2}(\pi_{\lambda,\beta})$
is also dense in the dual $((N)_{\pi_{\lambda,\beta}}^{\kappa})'=(N)_{\pi_{\lambda,\beta}}^{-\kappa}$.
Hence, we obtain a chain of densely embeddings $(N)_{\pi_{\lambda,\beta}}^{\kappa}\subset L^{2}(\pi_{\lambda,\beta})\subset(N)_{\pi_{\lambda,\beta}}^{-\kappa}$.
The characterization of these spaces is realized via integral transforms
and chain of spaces of entire functions of different types and order
of growth. Wick calculus extends in a straightforward manner from
Gaussian analysis to the present non-Gaussian framework. Finally,
in Appendix \ref{sec:Appell-Bell-connection} we give an explicit relation between (generalized) Appell
polynomials and Bell polynomials. 
\end{abstract}
\tableofcontents{}

\section{Introduction}

In 1830, the Poisson process was named after the French mathematician
Simeon Denis Poisson which describes the number of occurrences of
a certain event occurring in a given time period, given the average
number of times the event occurs over that time period. At present,
it is one of the most useful statistical distribution applied in a
wide range of fields, including astronomy, business, finance, medicine
and sports, just to name a few.

In 2000, O. N. Repin and A. I. Saichev \cite{Repin-Saichev00} initiated
the study of fractional Poisson process as a process with long-memory
effect which results from the non-exponential waiting time probability
distribution function. The fractional Poisson process is a natural
generalization of the Poisson process. In this case, a parameter $\beta,$
$0<\beta\leq1$, is introduced in the probability distribution function
of the fractional Poisson process wherein at $\beta=1$, the fractional
Poisson process becomes the Poisson process.

In this work, the starting point is the marginal distribution of the
fractional Poisson process as a measure on $\mathbb{N}_{0}$ or on
$\mathbb{R}$. Moreover, we investigate an abstract fractional Poisson
measure $\pi_{\lambda,\beta}$ which gives rise to a fractional Poisson
process $N_{\lambda,\beta}$ with time dependent rate. The two key
properties of the fractional Poisson measure are its analytic Laplace
transform $l_{\pi_{\lambda,\beta}}$ and to assign positive measure
to a nonempty set of its support, see Remark~\ref{rem:H2} below.
The polynomials associated with the Poisson measure are the well-known
orthogonal system of Charlier polynomials. However, the polynomials
associated with the fractional Poisson measure, a generalization to
the Charlier polynomials, are not orthogonal. Thus a biorthogonal
approach (see \cite{KSWY95,K97,KdSS98,GJRS14}) to the fractional
Poisson measure will be constructed, which involves the system of
Appell polynomials and the dual Appell system of the measure $\pi_{\lambda,\beta}$.
More precisely, the Appell polynomials are generated by the normalized
(or Wick) exponential 
\begin{equation}
\mathrm{e}_{\pi_{\lambda,\beta}}(z;x):=\frac{\mathrm{e}^{zx}}{l_{\pi_{\lambda,\beta}}(z)}=\frac{\mathrm{e}^{zx}}{E_{\beta}(\lambda(e^{z}-1))},\quad z\in\mathbb{C},\;x\in\mathbb{R},\label{eq:Wick-exponential}
\end{equation}
where $E_{\beta}$ is the Mittag-Leffler function, see \eqref{eq:ML-function}
for the definition. Interestingly, these Appell polynomials may be
expressed in terms of the Bell polynomials, see Appendix \ref{subsec:Connection-Appell-Bell}.
These polynomials play a key role in the construction of test function
spaces associated to $\pi_{\lambda,\beta}$. Below we use a modification
of $\mathrm{e}_{\pi_{\lambda,\beta}}(\cdot;x)$ (see \eqref{eq:Wick-expo})
adapted to the present framework, which produce the generalized Appell
polynomials. On the other hand, to construct and describe generalized
functions we need the so-called generalized dual Appell system. The
two systems thus introduced, that is, generalized Appell polynomials
and the generalized dual Appell system are biorthogonal with respect
to $L^{2}(\pi_{\lambda,\beta})$.

The paper is organized as follows. In Section \ref{sec:Fractional-Poisson-Measure},
we introduce the fractional Poisson measure. We show that the fractional
Poisson measure is a mixture of Poisson measures with a certain probability
measure $\nu_{\beta}$ on $\mathbb{R}_{+}$. In Section \ref{sec:Appell-polynomials},
we define and emphasize certain properties of the generalized Appell
polynomials. The generalized dual Appell system in then introduced
in Section \ref{sec:Dual-Appell-System} which forms a biorthogonal
system together with the generalized Appell polynomials, cf.~Proposition
\ref{thm:biorthogonal-property}. The construction of test and generalized
functions associated to $\pi_{\lambda,\beta}$ are presented in Section
\ref{sec:Test-Generalized-Function-Spaces} and in Section \ref{sec:Characterization-Theorems}
their characterization via integral transforms. Finally, in Section
\ref{sec:Wick-Calculus}, we introduce the Wick product (and related
calculus) in the larger space of generalized functions associated
to $\pi_{\lambda,\beta}$, that is, $(N)_{\pi_{\lambda,\beta}}^{-1}$.
In Appendix \ref{sec:Equivalent-norms} and \ref{sec:Appell-Bell-connection},
we show the equivalence of certain norms in the space of entire functions
(Appendix \ref{sec:Equivalent-norms}) and express the generalized
Appell polynomials in terms of Bell polynomials (Appendix \ref{sec:Appell-Bell-connection}).

\section{Fractional Poisson Measure}

\label{sec:Fractional-Poisson-Measure}The Poisson measure (probability
distribution) $\pi_{\lambda}$ on $\mathbb{N}_{0}:=\mathbb{N}\cup\{0\}$
(or $\mathbb{R})$ with rate $\lambda>0$ is defined on the $\sigma$-algebra
of all subsets of $\mathbb{N}$, denoted by $\mathscr{P}(\mathbb{N}_{0})$,
as
\[
\pi_{\lambda}(B):=\sum_{k\in B}\frac{\lambda^{k}}{k!}\mathrm{e}^{-\lambda}=\sum_{k=0}^{\infty}\frac{\lambda^{k}}{k!}\mathrm{e}^{-\lambda}\delta_{k}(B),\quad B\in\mathscr{P}(\mathbb{N}_{0}),
\]
where $\delta_{k}$ is the Dirac measure at $k\in\mathbb{N}_{0}$.
In particular, for every $k\in\mathbb{N}$ and $B=\{k\}\in\mathscr{P}(\mathbb{N}_{0})$,
we have
\[
\pi_{\lambda}(\{k\})=\frac{\lambda^{k}}{k!}\mathrm{e}^{-\lambda}.
\]
It is easy to obtain the Laplace transform $l_{\pi_{\lambda}}$ of
the measure $\pi_{\lambda}$, namely, for any $s\in\mathbb{R}$, we
have
\[
l_{\pi_{\lambda}}(s):=\int_{\mathbb{R}}\mathrm{e}^{sx}\,\mathrm{d}\pi_{\lambda}(x)=\mathrm{e}^{-\lambda}\sum_{k=0}^{\infty}\mathrm{e}^{sk}\frac{\lambda^{k}}{k!}=\exp\left(\lambda(\mathrm{e}^{s}-1)\right).
\]

\begin{rem}
The measure $\pi_{\lambda t}$, $t>0$, corresponds to the marginal
distribution of a standard Poisson process $N_{\lambda}=(N_{\lambda}(t))_{t\ge0}$
with parameter $\lambda t>0$ defined on a probability space $(\Omega,\mathcal{F},P)$.
More precisely, we have
\[
\pi_{\lambda t}(\{k\})=P(N_{\lambda}(t)=k)=\frac{(\lambda t)^{k}}{k!}\mathrm{e}^{-\lambda t},\quad k\in\mathbb{N}_{0}.
\]
The number $\pi_{\lambda t}(\{k\})$ is the probability that $k$
events occur in the time interval of length $t$.
\end{rem}

The Laplace transform $l_{\pi_{\lambda}}$ of $\pi_{\lambda}$ may
be extended to complex arguments $z\in\mathbb{C}$ and we obtain 
\begin{equation}
l_{\pi_{\lambda}}(z)=\int_{\mathbb{R}}\mathrm{e}^{zx}\,\mathrm{d}\pi_{\lambda}(x)=\mathrm{e}^{-\lambda}\sum_{k=0}^{\infty}\mathrm{e}^{zk}\frac{\lambda^{k}}{k!}=\exp\left(\lambda(\mathrm{e}^{z}-1)\right).\label{eq:LT-Pm}
\end{equation}

Now we would like to introduce the fractional Poisson measure (fPm).
At first we introduce the Mittag-Leffler function $E_{\beta}$ of
parameter $\beta\in(0,1]$. The Mittag-Leffler function is an entire
function defined on the complex plane by the power series 
\begin{equation}
E_{\beta}(z):=\sum_{n=0}^{\infty}\frac{z^{n}}{\Gamma(\beta n+1)},\quad z\in\mathbb{C}.\label{eq:ML-function}
\end{equation}
The Mittag-Leffler function plays the same role for the fPm as the
exponential function plays for Poisson measure.

For $0<\beta\le1$, the fractional Poisson measure $\pi_{\lambda,\beta}$
on $\mathbb{N}_{0}$ (or $\mathbb{R})$ with rate $\lambda>0$ is
defined for any $B\in\mathscr{P}(\mathbb{N}_{0})$ by
\[
\pi_{\lambda,\beta}(B):=\sum_{k\in B}\frac{\lambda^{k}}{k!}E_{\beta}^{(k)}(-\lambda)=\sum_{k=0}^{\infty}\frac{\lambda^{k}}{k!}E_{\beta}^{(k)}(-\lambda)\delta_{k}(B),
\]
where $E_{\beta}^{(k)}(z):=\frac{d^{k}}{dz^{k}}E_{\beta}(z)$ is the
$k$th derivative of the Mittag-Leffler $E_{\beta}$ function. In
particular, if $B=\{k\}\in\mathscr{P}(\mathbb{N}_{0})$, $k\in\mathbb{N}_{0}$,
we obtain
\[
\pi_{\lambda,\beta}(\{k\}):=\frac{\lambda^{k}}{k!}E_{\beta}^{(k)}(-\lambda).
\]
The Laplace transform of the measure $\pi_{\lambda,\beta}$ is given
for any $z\in\mathbb{C}$ by 
\begin{equation}
l_{\pi_{\lambda,\beta}}(z)=\int_{\mathbb{R}}\mathrm{e}^{zx}\,\mathrm{d}\pi_{\lambda,\beta}(x)=\sum_{k=0}^{\infty}\frac{\big(\mathrm{e}^{z}\lambda\big)^{k}}{k!}E_{\beta}^{(k)}\big(-\lambda\big)=E_{\beta}\big(\lambda(\mathrm{e}^{z}-1)\big).\label{eq:LT-fPm}
\end{equation}

\begin{rem}
The measure $\pi_{\lambda t^{\beta},\beta}$ corresponds to the marginal
distribution of the fractional Poisson process $N_{\lambda,\beta}=(N_{\lambda,\beta}(t))_{t\ge0}$
with rate $\lambda t^{\beta}>0$ defined on a probability space $(\Omega,\mathcal{F},P)$.
Thus, we obtain
\[
\pi_{\lambda t^{\beta}}(\{k\})=P(N_{\lambda,\beta}(t)=k)=\frac{(\lambda t^{\beta})^{k}}{k!}E_{\beta}^{(k)}(-\lambda t^{\beta}),\quad k\in\mathbb{N}_{0}.
\]
\end{rem}

\begin{rem}
The fractional Poisson process $N_{\lambda,\beta}$ was proposed by
O.\ N.\ Repin and A.\ I.\ Saichev \cite{Repin-Saichev00}. Since
then, it was studied and applied by many authors, see N.\ Laskin
\cite{L03}, F.\ Mainardi et al.\ \cite{MGS04,Mainardi-Gorenflo-Vivoli-05,Gorenflo2015},
V.\ V.\ Uchaikin et al.\ \cite{Uchaikin2008}, L.\ Beghin and
E.\ Orsingher \cite{Beghin:2009fi} M.\ Politi et al.\ \cite{Politi-Kaizoji-2011},
M.\ M.\ Meerschaert et al.\ \cite{Meerschaert2011}, R.\ Biard
and B.\ J.\ Saussereau \cite{Biard2014} and references therein.
\end{rem}

An interesting property of the fPm $\pi_{\lambda,\beta}$ is given
by a relation with the Poisson measure $\pi_{\lambda}$, namely, $\pi_{\lambda,\beta}$
is a mixture of Poisson measures with respect to a probability measure
$\nu_{\beta}$ on $\mathbb{R}_{+}:=[0,\infty)$. That probability
measure $\nu_{\beta}$ is absolutely continuous with respect to the
Lebesgue measure on $\mathbb{R}_{+}$ with a probability density $M_{\beta}$.
The Laplace transform of the density $M_{\beta}$ is given by, (see
\cite[Eq.~4.10]{Mainardi_Mura_Pagnini_2010} or \cite[Cor.~A.5]{GJRS14})
\[
\int_{0}^{\infty}\mathrm{e}^{-z\tau}\,\mathrm{d}\nu_{\beta}(\tau)=\int_{0}^{\infty}\mathrm{e}^{-z\tau}M_{\beta}(\tau)\,\mathrm{d}\tau=E_{\beta}(-z),\quad\forall z\in\mathbb{C}.
\]
We have the following lemma which gives the fPm as a mixture of Poisson
measures with different rates.
\begin{lem}
For $0<\beta\leq1$, the fPm $\pi_{\lambda,\beta}$ is an integral
(or mixture) of Poisson measure $\pi_{\lambda}$ with respect to the
probability measure $\nu_{\beta}$, i.e.,
\begin{equation}
\pi_{\lambda,\beta}=\int_{0}^{\infty}\pi_{\lambda\tau}\,\mathrm{d\nu_{\beta}(\tau),}\quad\forall\lambda>0.\label{eq:fPm-mixture}
\end{equation}
\end{lem}

\begin{proof}
Denote the right hand side of \eqref{eq:fPm-mixture} by $\mathrm{\mu}:=\int_{0}^{\infty}\mathrm{\pi}_{\lambda\tau}M_{\beta}(\tau)\mathrm{d}\tau$.
We compute the Laplace transform of $\mu$ and use Fubini's theorem
to obtain 
\begin{align*}
\int_{0}^{\infty}\mathrm{e}^{zx}\,\mathrm{d}\mu(x) & =\int_{0}^{\infty}\mathrm{e}^{zx}\int_{0}^{\infty}\mathrm{d\pi}_{\lambda\tau}(x)M_{\beta}(\tau)\mathrm{d}\tau\\
 & =\int_{0}^{\infty}\left(\int_{0}^{\infty}\mathrm{e}^{zx}\mathrm{d\pi}_{\lambda\tau}(x)\right)M_{\beta}(\tau)\mathrm{d}\tau\\
 & =\int_{0}^{\infty}\mathrm{e}^{\tau\lambda(\mathrm{e}^{z}-1)}M_{\beta}(\tau)\mathrm{d}\tau\\
 & =E_{\beta}(\lambda(\mathrm{e}^{z}-1)).
\end{align*}
Thus, we conclude that both the Laplace transforms of $\mu$ and $\pi_{\lambda,\beta}$
(cf.~\eqref{eq:LT-fPm}) coincides. The result follows by the uniqueness
of the Laplace transform.
\end{proof}
\begin{thm}[Moments of $\pi_{\lambda,\beta},$ cf.~\cite{L09}]
\label{thm:fPmm}The fPm $\pi_{\lambda,\beta}$ has moments of all
order. More precisely, the $n$th moment of the measure $\pi_{\lambda,\beta}$
is given by 
\begin{equation}
M_{\lambda,\beta}(n):=\int_{\mathbb{R}}x^{n}\,\mathrm{d}\pi_{\lambda,\beta}(x)=\sum_{m=0}^{n}\frac{m!}{\Gamma(m\beta+1)}S(n,m)\lambda^{m},\label{cnn}
\end{equation}
where $S(n,m)$ is the Stirling number of the second kind, see Definition
\ref{def:Stirling-numbers-second-kind}, \vpageref{def:Stirling-numbers-second-kind}.
\end{thm}

Here are the first few moments of the measure $\pi_{\lambda,\beta}$,
\begin{center}
$\begin{array}{rcl}
M_{\lambda,\beta}(0) & = & 1,\\
M_{\lambda,\beta}(1) & = & {\displaystyle \frac{\lambda}{\Gamma(\beta+1)}},\\
M_{\lambda,\beta}(2) & = & {\displaystyle \frac{\lambda}{\Gamma(\beta+1)}+\frac{2\lambda^{2}}{\Gamma(2\beta+1)},}\\
M_{\lambda,\beta}(3) & = & {\displaystyle \frac{\lambda}{\Gamma(\beta+1)}+\frac{6\lambda^{2}}{\Gamma(2\beta+1)}+\frac{6\lambda^{3}}{\Gamma(3\beta+1)},}\\
M_{\lambda,\beta}(4) & = & {\displaystyle \frac{\lambda}{\Gamma(\beta+1)}+\frac{14\lambda^{2}}{\Gamma(2\beta+1)}+\frac{36\lambda^{3}}{\Gamma(3\beta+1)}+\frac{24\lambda^{4}}{\Gamma(4\beta+1)}.}
\end{array}$
\par\end{center}

\begin{flushleft}
When $\beta=1$, these moments become the moments of the Poisson measure
$\pi_{\lambda}$:
\par\end{flushleft}

\begin{center}
$\begin{array}{rcl}
M_{\lambda,1}(0) & = & 1,\\
M_{\lambda,1}(1) & = & \lambda,\\
M_{\lambda,1}(2) & = & \lambda+\lambda^{2},\\
M_{\lambda,1}(3) & = & \lambda+3\lambda^{2}+\lambda^{3},\\
M_{\lambda,1}(4) & = & \lambda+7\lambda^{2}+6\lambda^{3}+\lambda^{4}.
\end{array}$
\par\end{center}

\section{Generalized Appell Polynomials}

\label{sec:Appell-polynomials}In this section we introduce the system
of generalized Appell polynomials associated with the fPm $\pi_{\lambda,\beta}$
in $\mathbb{R}$. Our first concern is the analytic property of the
Laplace transform given in \eqref{eq:LT-fPm}, that is, 
\[
l_{\pi_{\lambda,\beta}}(z)=E_{\beta}(\lambda(\mathrm{e}^{z}-1)),\quad z\in\mathbb{C}.
\]
In fact, $l_{\pi_{\lambda,\beta}}(\cdot)$ is the composition of two
entire functions, thus it is entire.

The analytic property of the function $l_{\pi_{\lambda,\beta}}(\cdot)$
is equivalent and characterized by the following proposition.
\begin{prop}
\label{prop:Analytic-LT-fPm}Let $\pi_{\lambda,\beta}$ be the fPm
in $\mathbb{R}$. The following statements are equivalent:
\begin{enumerate}
\item $\exists C,K>0$ such that $\forall n\in\mathbb{N}_{0}$, $|\int_{\mathbb{R}}x^{n}\,\mathrm{d}\pi_{\lambda,\beta}(x)|<n!C^{n}K,$
\item $\exists\varepsilon>0$, such that $\int_{\mathbb{R}}\mathrm{e}^{\varepsilon|x|}\,\mathrm{d}\pi_{\lambda,\beta}(x)<\infty.$
\end{enumerate}
\end{prop}

\begin{proof}
\begin{description}
\item [{1.~$\Rightarrow$~2.}] Let $\pi_{\lambda,\beta}$ be the fPm
in $\mathbb{R}$. The Taylor expansion of $l_{\pi_{\lambda,\beta}}(\cdot)$
is given by 
\[
l_{\pi_{\lambda,\beta}}(z)=\sum_{n=0}^{\infty}\frac{z^{n}}{n!}\left[\frac{\mathrm{d}^{n}}{\mathrm{d}z^{n}}l_{\pi_{\lambda,\beta}}(z)\right]_{z=0},\quad z\in\mathbb{C}.
\]
Then by definition of $l_{\pi_{\lambda,\beta}}(\cdot)$, we have 
\[
l_{\pi_{\lambda,\beta}}(z)=\int_{\mathbb{R}}\mathrm{e}^{zx}\,\mathrm{d}\pi_{\lambda,\beta}(x)=\sum_{n=0}^{\infty}\frac{z^{n}}{n!}\int_{\mathbb{R}}x^{n}\,\mathrm{d}\pi_{\lambda,\beta}(x).
\]
Now, note that 
\[
\int_{\mathbb{R}}\mathrm{e}^{\varepsilon|x|}\,\mathrm{d}\pi_{\lambda,\beta}(x)=\sum_{n=0}^{\infty}\frac{\varepsilon^{n}}{n!}\int_{\mathbb{R}}|x|^{n}\,\mathrm{d}\pi_{\lambda,\beta}(x).
\]
By the Cauchy-Schwarz inequality, we have 
\[
\int_{\mathbb{R}}|x^{n}|\,\mathrm{d}\pi_{\lambda,\beta}(x)\leq\left(\int_{\mathbb{R}}x^{2n}\,\mathrm{d}\pi_{\lambda,\beta}(x)\right)^{1/2}
\]
and so 
\[
\int_{\mathbb{R}}\mathrm{e}^{\varepsilon|x|}\,\mathrm{d}\pi_{\lambda,\beta}(x)\leq\sum_{n=0}^{\infty}\frac{\varepsilon^{n}}{n!}\left(\int_{\mathbb{R}}x^{2n}\,\mathrm{d}\pi_{\lambda,\beta}(x)\right)^{1/2}.
\]
By the hypothesis, we have 
\[
\int_{\mathbb{R}}x^{n}\,\mathrm{d}\pi_{\lambda,\beta}(x)\leq\left|\int_{\mathbb{R}}x^{n}\,\mathrm{d}\pi_{\lambda,\beta}(x)\right|<n!C^{n}K
\]
for all $n\in\mathbb{N}_{0}$, therefore 
\[
\int_{\mathbb{R}}x^{2n}\,\mathrm{d}\pi_{\lambda,\beta}(x)\leq(2n)!C^{2n}K
\]
Since $(2n)!\leq2^{2n}(n!)^{2}$ for all $n\in\mathbb{N}_{0}$, we
have 
\[
\int_{\mathbb{R}}x^{2n}\,\mathrm{d}\pi_{\lambda,\beta}(x)\leq2^{2n}(n!)^{2}C^{2n}K.
\]
Thus, 
\[
\begin{array}{rcl}
{\displaystyle \int_{\mathbb{R}}\mathrm{e}^{\varepsilon|x|}\,\mathrm{d}\pi_{\lambda,\beta}(x)} & \leq & {\displaystyle \sum_{n=0}^{\infty}\frac{\varepsilon^{n}}{n!}\left(\int_{\mathbb{R}}x^{2n}\,\mathrm{d}\pi_{\lambda,\beta}(x)\right)^{1/2}}\\
 & \leq & {\displaystyle \sum_{n=0}^{\infty}\frac{\varepsilon^{n}}{n!}2^{n}(n!)C^{n}\sqrt{K}}\\
 & = & \sqrt{K}{\displaystyle \sum_{n=0}^{\infty}(2C\varepsilon)^{n}}
\end{array}
\]
which is finite provided $\varepsilon$ is chosen such that $2C\varepsilon<1.$
\item [{2.~$\Rightarrow$~1.}] Suppose that there exists $\varepsilon>0$
and $K_{\varepsilon}>0$ such that 
\[
\int_{\mathbb{R}}\mathrm{e}^{\varepsilon|x|}\,\mathrm{d}\pi_{\lambda,\beta}(x)=K_{\varepsilon}.
\]
Then 
\[
\sum_{n=0}^{\infty}\frac{\varepsilon^{n}}{n!}\int_{\mathbb{R}}|x|^{n}\,\mathrm{d}\pi_{\lambda,\beta}(x)=K_{\varepsilon}.
\]
This implies that each term in this series is less than or equal to
$K_{\varepsilon}$, that is 
\[
\frac{\varepsilon^{n}}{n!}\int_{\mathbb{R}}|x|^{n}\,\mathrm{d}\pi_{\lambda,\beta}(x)\leq K_{\varepsilon}.
\]
Hence, we have 
\[
\left|\int_{\mathbb{R}}x^{n}\,\mathrm{d}\pi_{\lambda,\beta}(x)\right|\leq\int_{\mathbb{R}}|x|^{n}\,\mathrm{d}\pi_{\lambda,\beta}(x)\leq n!\left(\frac{1}{\varepsilon^{n}}\right)K_{\varepsilon}.
\]
Now, take $C=\frac{1}{\varepsilon}$ and $K=K_{\varepsilon}$ and
we are done.\hfill{}$\qedhere$
\end{description}
\end{proof}
Now we emphasize the condition on the measure $\pi_{\lambda,\beta}$
which guarantees the embedding of test function spaces (to be introduced
in Section~\ref{sec:Test-Generalized-Function-Spaces}) in $L^{2}(\pi_{\lambda,\beta})$,
that is, the space of complex-valued measurable functions whose modulus
square is integrable with respect to $\pi_{\lambda,\beta}$, see \cite{KaKo99}.
\begin{rem}
\label{rem:H2} For every nonempty open set $O\subset\mathbb{R}$
such that $\mathbb{N}\cap O\neq\varnothing$ we have $\pi_{\lambda,\beta}(O)>0.$
\end{rem}

We define the following entire function $f$ on a neighborhood of
$0\in\mathbb{C}$ by
\[
f(z):=\log(1+z).
\]
Recall from \eqref{eq:Wick-exponential}, the \emph{Wick exponential}
with respect to the measure $\pi_{\lambda,\beta}$ defined by 
\begin{equation}
\mathrm{e}_{\pi_{\lambda,\beta}}(z;\cdot):\mathbb{R}\longrightarrow\mathbb{C},\;x\mapsto\mathrm{e}_{\pi_{\lambda,\beta}}(z;x)={\displaystyle \frac{\mathrm{e}^{xz}}{l_{\pi_{\lambda,\beta}}(z)}}={\displaystyle \frac{\mathrm{e}^{xz}}{E_{\beta}(\lambda(\mathrm{e}^{z}-1))}}.\label{eq:Wick-gen-Appell}
\end{equation}

Since $l_{\pi_{\lambda,\beta}}(0)=1$, there is a neighborhood $V$
of $0\in\mathbb{C}$ where $\mathrm{e}_{\pi_{\lambda,\beta}}(f(\cdot);x)$
is given by 
\begin{equation}
\mathrm{e}_{\pi_{\lambda,\beta}}(f(z);x)=\sum_{n=0}^{\infty}\frac{z^{n}}{n!}C_{n}^{\lambda,\beta}(x),\hspace{0.1in}\forall z\in V,\label{eq:Wick-expo}
\end{equation}
with 
\begin{equation}
C_{n}^{\lambda,\beta}(x)=\frac{\mathrm{d}^{n}}{\mathrm{d}z^{n}}\mathrm{e}_{\pi_{\lambda,\beta}}(f(z);x)\Big|_{z=0}.\label{eq:Appell-Polynomials}
\end{equation}
These functions $C_{n}^{\lambda,\beta}(\cdot)$, $n\in\mathbb{N}$,
are polynomials of degree $n$ which are also known as the \textit{generalized
Appell polynomials}. The set $\{C_{n}^{\lambda,\beta}(\cdot),n\in\mathbb{N}_{0}\}$
is called the \textit{generalized Appell polynomial system} associated
with the fPm $\pi_{\lambda,\beta}$ and we denote it by $\mathbb{P}^{\pi_{\lambda,\beta}}$.
The following are the first four generalized Appell polynomials $C_{n}^{\lambda,\beta}(x)$
associated with the fPm:

$\begin{array}{ccl}
C_{0}^{\lambda,\beta}(x) & = & 1,\\
C_{1}^{\lambda,\beta}(x) & = & {\displaystyle x-\frac{\lambda}{\Gamma(\beta+1)},}\\
C_{2}^{\lambda,\beta}(x) & = & {\displaystyle x^{2}-\left(\frac{2\lambda}{\Gamma(\beta+1)}+1\right)x-\frac{2\lambda^{2}}{\Gamma(2\beta+1)}+2\left(\frac{\lambda}{\Gamma(\beta+1)}\right)}^{2}\\
C_{3}^{\lambda,\beta}(x) & = & {\displaystyle x^{3}-3\left(\frac{\lambda}{\Gamma(\beta+1)}+1\right)x^{2}+\Bigg[6\left(\frac{\lambda}{\Gamma(\beta+1)}\right)^{2}-\frac{6\lambda^{2}}{\Gamma(2\beta+1)}+\frac{3\lambda}{\Gamma(\beta+1)}+2\Bigg]x}\\
 &  & {\displaystyle -\frac{6\lambda^{3}}{\Gamma(3\beta+1)}{\displaystyle +\frac{12\lambda^{3}}{\Gamma(\beta+1)\Gamma(2\beta+1)}}{\displaystyle -6\left(\frac{\lambda}{\Gamma(\beta+1)}\right)^{3}.}}
\end{array}$

At $\beta=1$ these polynomials become the well known Charlier polynomials,
i.e.,
\begin{equation}
\begin{split}C_{0}^{\lambda,1}(x) & =1\\
C_{1}^{\lambda,1}(x) & =x-\lambda\\
C_{2}^{\lambda,1}(x) & =x^{2}-(1+2\lambda)x+\lambda^{2}\\
C_{3}^{\lambda,1}(x) & =x^{3}-(3+3\lambda)x^{2}+(2+3\lambda+3\lambda^{2})x-\lambda^{3}.
\end{split}
\label{eq:Charlier-polynomials}
\end{equation}
Later on, we consider the Taylor series
\begin{equation}
\mathrm{e}_{\pi_{\lambda,\beta}}(z;x)=\frac{\mathrm{e}^{zx}}{l_{\pi_{\lambda,\beta}}(z)}=\sum_{n=0}^{\infty}\frac{z^{n}}{n!}A_{n}^{\lambda,\beta}(x),\label{eq:Wick-Apell}
\end{equation}
for all $z\in W$, a neighborhood of zero in $\mathbb{C}$. The functions
$A_{n}^{\lambda,\beta}(\cdot)$, $n\in\mathbb{N}$, are polynomials
of degree $n$ which are also known as the \textit{Appell polynomials}.
In Appendix \ref{subsec:Connection-Appell-Bell}, we obtain the explicit
form of the Appell polynomials generated by the fPm in terms of the
Bell polynomials, see Theorem \ref{thm:Appell-Bell-relation}.

The next proposition summarizes the most important properties of the
polynomials $C_{n}^{\lambda,\beta}(\cdot)$, $n\in\mathbb{N}_{0}$.
\begin{prop}
\label{prop:gen-Appell-prop} For any $x,y\in\mathbb{R}$, the polynomials
$C_{n}^{\lambda,\beta}(\cdot)$, $n\in\mathbb{N}$, satisfy the following
properties
\begin{description}
\item [{(P1)}] $C_{n}^{\lambda,\beta}(x)={\displaystyle \sum_{m=0}^{n}\frac{A_{n}^{m}}{m!}A_{m}^{\lambda,\beta}(x),}$
where $A_{0}^{0}:=1$ and $A_{n}^{m}:={\displaystyle \!\!\!\!\!\!\sum_{l_{1}+\dots+l_{m}=n}\!\!\!\!\frac{(-1)^{n+m}n!}{l_{1}\dots l_{m}}}$,
$l_{i}\in\left\{ 1,\dots,n\right\} $ for all $i=1,\dots,m.$
\item [{(P2)}] $x^{n}={\displaystyle \sum_{k=0}^{n}\sum_{m=0}^{k}\binom{n}{k}\frac{B_{k}^{m}}{m!}}C_{m}^{\lambda,\beta}(x)M_{\lambda,\beta}(n-k)$,
where $B_{0}^{0}:=1$ and $B_{k}^{m}:=\!\!\!\!{\displaystyle \sum_{l_{1}+\dots+l_{m}=k}\frac{k!}{l_{1}!\dots l_{m}!}}$,
$l_{i}\in\left\{ 1,\dots,k\right\} $ for all $i=1,\dots,m.$
\item [{(P3)}] $C_{n}^{\lambda,\beta}(x+y)=\!\!\!\!{\displaystyle \sum_{k+l+m=n}\frac{n!}{k!l!m!}}C_{k}^{\lambda,\beta}(x)C_{l}^{\lambda,\beta}(y)\tilde{M}_{\lambda,\beta}(m)$,
where $\tilde{M}_{\lambda,\beta}(m):={\displaystyle \frac{m!\lambda^{m}}{\Gamma(m\beta+1)}}$.
\item [{(P4)}] $C_{n}^{\lambda,\beta}(x+y)={\displaystyle \sum_{k=0}^{n}\binom{n}{k}}C_{k}^{\lambda,\beta}(x)(y)_{n-k}$,
where $(y)_{m}$ are the falling factorials, for all $m\in\mathbb{N}_{0}$.
\item [{(P5)}] $\mathbb{E}\big(C_{n}^{\lambda,\beta}(\cdot)\big)=\delta_{n,0}$,
where $\delta_{n,m}$ is the Kronecker delta and $\mathbb{E}(\cdot)$
is the expectation with respect to the measure $\pi_{\lambda,\beta}$.
\item [{(P6)}] For every $\varepsilon>0$, there exists $C_{\varepsilon},\sigma_{\varepsilon}>0$
such that 
\[
|C_{n}^{\lambda,\beta}(x)|\leq C_{\varepsilon}n!\sigma_{\varepsilon}^{-n}\mathrm{e}^{\varepsilon|x|}.
\]
\end{description}
\end{prop}

\begin{proof}
Let $C_{n}^{\lambda,\beta}(x)$, $n\in\mathbb{N}_{0}$, be the generalized
Appell polynomials generated by the measure $\pi_{\lambda,\beta}$
and let $x,y\in\mathbb{R}$ be given.
\begin{description}
\item [{(P1)}] In view of equation \eqref{eq:Wick-Apell}, we have
\[
{\displaystyle \mathrm{e}_{\pi_{\lambda,\beta}}(f(z);x):=}{\displaystyle \frac{\mathrm{e}^{xf(z)}}{l_{\pi_{\lambda,\beta}}(f(z))}}=\sum_{m=0}^{\infty}\frac{f(z)^{m}}{m!}A_{m}^{\lambda,\beta}(x),
\]
for any $z\in V$ (neighborhood of zero in $\mathbb{C}$). Denote
$A_{l}$ the polynomials of degree $l$ generated by $f(z)$, i.e.,
\[
f(z)=\log(1+z)=\sum_{l=0}^{\infty}\frac{z^{l}}{l!}A_{l}
\]
where 
\[
A_{l}=\frac{\mathrm{d}^{l}}{\mathrm{d}z^{l}}f(z)\Big|_{z=0}=(-1)^{l+1}(l-1)!
\]
for all $l>0$ and $A_{0}=0$. Hence, 
\begin{align*}
\sum_{n=0}^{\infty}\frac{z^{n}}{n!}C_{n}^{\lambda,\beta}(x) & =A_{0}^{\lambda,\beta}(x)+\sum_{m=1}^{\infty}\frac{1}{m!}\left(\sum_{l=0}^{\infty}\frac{z^{l}}{l!}A_{l}\right)^{m}A_{m}^{\lambda,\beta}(x)\\
 & =A_{0}^{\lambda,\beta}(x)+\sum_{m=1}^{\infty}\frac{1}{m!}\sum_{n=m}^{\infty}\frac{z^{n}}{n!}\sum_{l_{1}+\dots+l_{m}=n}\frac{n!}{l_{1}!\dots l_{m}!}\prod_{i=1}^{m}A_{l_{i}}A_{m}^{\lambda,\beta}(x)\\
 & =A_{0}^{\lambda,\beta}(x)+\sum_{m=1}^{\infty}\frac{1}{m!}\sum_{n=m}^{\infty}\frac{z^{n}}{n!}\underbrace{\sum_{l_{1}+\dots+l_{m}=n}\frac{(-1)^{n+m}n!}{l_{1}\dots l_{m}}}_{=:A_{n}^{m}}A_{m}^{\lambda,\beta}(x)\\
 & =\sum_{n=0}^{\infty}\frac{z^{n}}{n!}\sum_{m=0}^{n}\frac{A_{n}^{m}}{m!}A_{m}^{\lambda,\beta}(x),
\end{align*}
where $A_{0}^{0}=1$ and $l_{i}\in\left\{ 1,\dots,n\right\} $ for
all $i=1,\dots,m.$ The result follows by comparing both sides of
the equation.
\item [{(P2)}] Note that $g(z)=\mathrm{e}^{z}-1$ is the inverse of $f(z)=\log(1+z)$.
Similar as in (P1), we use equation \eqref{eq:Wick-gen-Appell} such
that for any $z\in V$ (neighborhood of zero in $\mathbb{C}$), we
have 
\[
{\displaystyle \frac{\mathrm{e}^{xz}}{l_{\pi_{\lambda,\beta}}(z)}}=\sum_{m=0}^{\infty}\frac{g(z)^{m}}{m!}C_{m}^{\lambda,\beta}(x).
\]
Note that
\[
g(z)=\mathrm{e}^{z}-1=\sum_{l=1}^{\infty}\frac{z^{l}}{l!}.
\]
Thus,
\begin{align*}
\sum_{n=0}^{\infty}\frac{z^{n}}{n!}A_{n}^{\lambda,\beta}(x) & =C_{0}^{\lambda,\beta}(x)+\sum_{m=1}^{\infty}\frac{1}{m!}\left(\sum_{l=1}^{\infty}\frac{z^{l}}{l!}\right)^{m}C_{m}^{\lambda,\beta}(x)\\
 & =C_{0}^{\lambda,\beta}(x)+\sum_{m=1}^{\infty}\frac{1}{m!}\sum_{n=m}^{\infty}\frac{z^{n}}{n!}\underbrace{\sum_{l_{1}+\dots+l_{m}=n}\frac{n!}{l_{1}!\dots l_{m}!}}_{=:B_{n}^{m}}C_{m}^{\lambda,\beta}(x)\\
 & =\sum_{n=0}^{\infty}\frac{z^{n}}{n!}\sum_{m=0}^{n}\frac{B_{n}^{m}}{m!}C_{m}^{\lambda,\beta}(x),
\end{align*}
where $B_{0}^{0}=1$ and $l_{i}\in\left\{ 1,\dots,n\right\} $ for
all $i=1,\dots,m.$ Comparing both sides of the equation, we obtain
\begin{equation}
A_{n}^{\lambda,\beta}(x)=\sum_{m=0}^{n}\frac{B_{n}^{m}}{m!}C_{m}^{\lambda,\beta}(x).\label{eq:appell-gen-appell-formula}
\end{equation}
Now, we use the equality for any $z\in W$ (neighborhood of zero in
$\mathbb{C}$), 
\[
\mathrm{e}^{xz}=\mathrm{e}_{\pi_{\lambda,\beta}}(z;x)l_{\pi_{\lambda,\beta}}(z)
\]
so that 
\begin{align*}
\sum_{n=0}^{\infty}\frac{x^{n}}{n!}z^{n} & =\sum_{m=0}^{\infty}\frac{z^{m}}{m!}A_{m}^{\lambda,\beta}(x)\cdot\sum_{k=0}^{\infty}\frac{z^{k}}{k!}M_{\lambda,\beta}(n)\\
 & =\sum_{n=0}^{\infty}\frac{z^{n}}{n!}\left(\sum_{k=0}^{n}\binom{n}{k}A_{k}^{\lambda,\beta}(x)M_{\lambda,\beta}(n-k)\right).
\end{align*}
This implies that 
\begin{equation}
z^{n}=\sum_{k=0}^{n}\binom{n}{k}A_{k}^{\lambda,\beta}(x)M_{\lambda,\beta}(n-k).\label{eq:Appell-z^n}
\end{equation}
The result follows by applying equation \eqref{eq:appell-gen-appell-formula}
to \eqref{eq:Appell-z^n}.
\item [{(P3)}] By definition of the Wick exponential,
\[
\mathrm{e}_{\pi_{\lambda,\beta}}(f(z);x+y)=\mathrm{e}_{\pi_{\lambda,\beta}}(f(z);x)\mathrm{e}_{\pi_{\lambda,\beta}}(f(z);y)l_{\pi_{\lambda,\beta}}(f(z)).
\]
The Taylor expansion of $l_{\pi_{\lambda,\beta}}(f(z))$ around $z=0$
is given by
\begin{align*}
l_{\pi_{\lambda,\beta}}(f(z)) & =E_{\beta}(\lambda z)=\sum_{m=0}^{\infty}\frac{(\lambda z)^{m}}{\Gamma(m\beta+1)}\\
 & =\sum_{m=0}^{\infty}\frac{z^{m}}{m!}\underbrace{\frac{\lambda^{m}m!}{\Gamma(m\beta+1)}}_{=:\tilde{M}_{\lambda,\beta}(m)}=\sum_{m=0}^{\infty}\frac{z^{m}}{m!}\tilde{M}_{\lambda,\beta}(m).
\end{align*}
Then, we have 
\begin{align*}
\sum_{n=0}^{\infty}\frac{z^{n}}{n!}C_{n}^{\lambda,\beta}(x+y) & =\sum_{k=0}^{\infty}\frac{z^{k}}{k!}C_{k}^{\lambda,\beta}(x)\cdot\sum_{l=0}^{\infty}\frac{z^{l}}{l!}C_{l}^{\lambda,\beta}(y)\cdot\sum_{m=0}^{\infty}\frac{z^{m}}{m!}\tilde{M}_{\lambda,\beta}(m)\\
 & =\sum_{n=0}^{\infty}\frac{z^{n}}{n!}{\displaystyle \sum_{k+l+m=n}\frac{n!}{k!l!m!}}C_{k}^{\lambda,\beta}(x)C_{l}^{\lambda,\beta}(y)\tilde{M}_{\lambda,\beta}(m).
\end{align*}
The result follows by comparing the coefficients in both sides of
the equation.
\item [{(P4)}] Again, by definition of the Wick exponential, 
\[
\mathrm{e}_{\pi_{\lambda,\beta}}(f(z);x+y)=\mathrm{e}_{\pi_{\lambda,\beta}}(f(z);x)\exp(yf(z)).
\]
The Taylor expansion of $\exp(yf(z))$ around $z=0$ is given by 
\[
\exp(yf(z))=\sum_{m=0}^{\infty}\frac{z^{m}}{m!}(y)_{m},
\]
where $(y)_{m}$ are the falling factorials (see for example \cite{Gradshteyn2014})
given by 
\[
(y)_{m}:=\begin{cases}
1 & \mathrm{if}\,m=0\\
y(y-1)\dots(y-m+1) & \mathrm{if}\,m\in\mathbb{N}.
\end{cases}
\]
 Hence, 
\begin{align*}
\sum_{n=0}^{\infty}\frac{z^{n}}{n!}C_{n}^{\lambda,\beta}(x+y) & =\sum_{k=0}^{\infty}\frac{z^{k}}{k!}C_{k}^{\lambda,\beta}(x)\cdot\sum_{m=0}^{\infty}\frac{z^{m}}{m!}(y)_{m}\\
 & =\sum_{n=0}^{\infty}\frac{z^{n}}{n!}{\displaystyle \sum_{k+m=n}\frac{n!}{k!m!}}C_{k}^{\lambda,\beta}(x)(y)_{m}\\
 & =\sum_{n=0}^{\infty}\frac{z^{n}}{n!}{\displaystyle \sum_{k=0}^{n}\binom{n}{k}}C_{k}^{\lambda,\beta}(x)(y)_{n-k}.
\end{align*}
The claim follows again by comparing the coefficients in both sides
of the equation.
\item [{(P5)}] Note that
\[
\mathbb{E}\big(\mathrm{e}_{\pi_{\lambda,\beta}}(f(z);\cdot)\big)=\frac{1}{l_{\pi_{\lambda,\beta}}(f(z))}\mathbb{E}(\mathrm{e}^{f(z)\cdot})=\frac{l_{\pi_{\lambda,\beta}}(f(z))}{l_{\pi_{\lambda,\beta}}(f(z))}=1
\]
and using \eqref{eq:Wick-expo} we obtain
\[
\mathbb{E}\big(\mathrm{e}_{\pi_{\lambda,\beta}}(f(z);\cdot)\big)=\sum_{n=0}^{\infty}\frac{z^{n}}{n!}\int_{\mathbb{R}}C_{n}^{\lambda,\beta}(x)\,\mathrm{d}\pi_{\lambda,\beta}(x)
\]
which implies the result comparing the coefficients.
\item [{(P6)}] Given $\varepsilon>0$ let $\sigma_{\varepsilon}>0$ be
such that $|f(z)|<\varepsilon$ for any $z\in\{z'\mid|z'|=\sigma_{\varepsilon}\}$.
The following bound follows from the definition of the polynomials
$C_{n}^{\lambda,\beta}(\cdot)$, $n\in\mathbb{N}$ and the Cauchy
inequality 
\begin{align*}
\big|C_{n}^{\lambda,\beta}(x)\big| & \le\frac{n!}{2\pi}\int_{|z|=\sigma_{\varepsilon}}\frac{|\mathrm{e}_{\pi_{\lambda,\beta}}(f(z);x)|}{|z|^{n+1}}\,|\mathrm{d}z|\\
 & \le\frac{n!}{2\pi}\int_{|z|=\sigma_{\varepsilon}}\frac{\mathrm{e}^{|f(z)||x|}}{|z|^{n+1}|l_{\pi_{\lambda,\beta}}(f(z))|}\,|\mathrm{d}z|\\
 & \le n!\underbrace{\sup_{|z|=\sigma_{\varepsilon}}\frac{1}{|l_{\pi_{\lambda,\beta}}(f(z))|}}_{=:C_{\varepsilon}}\frac{\mathrm{e}^{\varepsilon|x|}}{\sigma_{\varepsilon}^{n}}\\
 & =C_{\varepsilon}n!\sigma_{\varepsilon}^{-n}\mathrm{e}^{\varepsilon|x|}.
\end{align*}
This concludes the proof.\hfill{}$\qedhere$
\end{description}
\end{proof}
An alternative representation of the properties (P1) and (P2) of $C_{n}^{\lambda,\beta}(\cdot)$
from Proposition \ref{prop:gen-Appell-prop} is given in the following
corollary.
\begin{cor}
\label{cor:gen-Appell-polynomials}For any $x,y\in\mathbb{R}$, the
polynomials $C_{n}^{\lambda,\beta}(\cdot)$, $n\in\mathbb{N}$, satisfy
the following properties
\begin{description}
\item [{(P1$'$)}] $C_{n}^{\lambda,\beta}(x)={\displaystyle \sum_{m=0}^{n}s(n,m)A_{m}^{\lambda,\beta}(x),}$
where $s(m,n)$ is the Stirling numbers of the first kind.
\item [{(P2$'$)}] $x^{n}={\displaystyle \sum_{k=0}^{n}\sum_{m=0}^{k}\binom{n}{k}}C_{m}^{\lambda,\beta}(x)S(k,m)M_{\lambda,\beta}(n-k)$,
where $S(m,n)$ is the Stirling numbers of the second kind.
\end{description}
\end{cor}

\begin{proof}
The claim in (P1$'$) (resp. (P2$'$)) follows directly from Proposition
\ref{prop:gen-Appell-prop}--(P1) (resp. (P2)) and Proposition~\ref{prop:Stirling-numbers-explicit}--1
(resp. 2) in Appendix~\ref{subsec:Stirling-Numbers}. 
\end{proof}
\begin{rem}
In Appendix \ref{subsec:Connection-Appell-Bell}, we also obtain the
explicit form of the generalized Appell polynomials generated by the
fractional Poisson measure in terms of the Bell polynomials, see Theorem
\ref{thm:gen-Appell-Bell-relation}.
\end{rem}

\section{Generalized Dual Appell System}

\label{sec:Dual-Appell-System}Let us consider the Hilbert space $L^{2}(\pi_{\lambda,\beta}):=L^{2}(\mathbb{R},\mathcal{B}(\mathbb{R}),\pi_{\lambda,\beta};\mathbb{C})$
of all complex-valued measurable functions whose modulus square is
integrable with respect to $\pi_{\lambda,\beta}$. In addition, we
denote by $\mathcal{P}(\mathbb{R})$ the set of all polynomials in
$\mathbb{R}$ with complex coefficients. Using the Proposition~\ref{prop:gen-Appell-prop}--(P2),
we can write $\mathcal{P}(\mathbb{R})$ in the form 
\[
\mathcal{P}(\mathbb{R})=\left\{ \varphi:\mathbb{R}\rightarrow\mathbb{C}\mid\varphi(x)=\sum_{k=0}^{n}\varphi_{k}C_{k}^{\lambda,\beta}(x),~\varphi_{k}\in\mathbb{C},~n\in\mathbb{N}_{0}\right\} .
\]
Since $\pi_{\lambda,\beta}$ admits the moments of all orders, then
we have $\mathcal{P}(\mathbb{R})\subset L^{2}(\pi_{\lambda,\beta})$.
On the other hand, Remark \ref{rem:H2} ensures that the inclusion
$\mathcal{P}(\mathbb{R})\hookrightarrow L^{2}(\pi_{\lambda,\beta})$
is dense.

In $\mathcal{P}(\mathbb{R})$ the following notion of convergence
is defined: a sequence of polynomials $(\varphi_{i})_{i\in\mathbb{N}}\subset\mathcal{P}(\mathbb{R})$
converges to $\varphi\in\mathcal{P}(\mathbb{R})$ if and only if the
sequence $(n_{\varphi_{i}})_{i\in\mathbb{N}}$ of the degree of the
polynomials $\varphi_{i}$ is bounded and the coefficients converge
term by term.

Let us consider the differential operator of the order $k$, $k\in\mathbb{N}_{0}$,
denoted by $\nabla^{k}$, defined in $\mathcal{P}(\mathbb{R})$ by
\[
(\nabla^{k}\varphi)(x)={\displaystyle \frac{\mathrm{d}^{k}\varphi(x)}{\mathrm{d}x^{k}}},\quad\varphi\in\mathcal{P}(\mathbb{R}).
\]
Note that $\nabla^{k}$ is a continuous linear operator with respect
to the above convergence.

Let $\mathcal{P}'_{\pi_{\lambda,\beta}}(\mathbb{R})$ be the dual
of $\mathcal{P}(\mathbb{R})$ with respect to $L^{2}(\pi_{\lambda,\beta}),$
that is, $\mathcal{P}'_{\pi_{\lambda,\beta}}(\mathbb{R})$ is the
set of all continuous linear functionals defined in $\mathcal{P}(\mathbb{R})$
which are given in terms of the inner product in $L^{2}(\pi_{\lambda,\beta})$.
More precisely, for every $\Phi\in L^{2}(\pi_{\lambda,\beta})\subset\mathcal{P}'_{\pi_{\lambda,\beta}}(\mathbb{R})$
and any $\varphi\in\mathcal{P}(\mathbb{R})$ we use the notation,
\[
\Phi(\varphi):=\langle\!\langle\varphi,\Phi\rangle\!\rangle_{\pi_{\lambda,\beta}}:=(\!(\varphi,\bar{\Phi})\!)_{L^{2}(\pi_{\lambda,\beta})}.
\]
The bilinear map $\langle\!\langle\cdot,\cdot\rangle\!\rangle_{\pi_{\lambda,\beta}}$
is also called \textit{dual pair} between $\mathcal{P}'_{\pi_{\lambda,\beta}}(\mathbb{R})$
and $\mathcal{P}(\mathbb{R})$. The elements of $\mathcal{P}'_{\pi_{\lambda,\beta}}(\mathbb{R})$
are called \textit{generalized functions}. So, we obtain a chain of
continuous and dense embeddings
\[
\mathcal{P}(\mathbb{R})\subset L^{2}(\pi_{\lambda,\beta})\subset\mathcal{P}'_{\pi_{\lambda,\beta}}(\mathbb{R}).
\]
Consider the function $g(z)=\mathrm{e}^{z}-1$ which is the inverse
of $f(z)=\log(1+z)$. For $k\in\mathbb{N}$, recall that the Taylor
expansion of $\left(g(z)\right)^{k}$ and $\left(f(z)\right)^{k}$
at $z=0$ are given by the following series
\[
g(z)^{k}=k!\sum_{n=k}^{\infty}S(n,k)\frac{z^{n}}{n!}\quad\mathrm{and}\quad f(z)^{k}=k!\sum_{n=k}^{\infty}s(n,k)\frac{z^{n}}{n!},
\]
where $s(n,k)$ and $S(n,k)$ are the Stirling numbers of the first
kind and Stirling numbers of the second kind, respectively. Then,
for any $\varphi\in\mathcal{P}(\mathbb{R})$, we have
\[
(g(\nabla)^{k}\varphi)(x)=k!\sum_{n=k}^{\infty}\frac{S(n,k)}{n!}\frac{\mathrel{\mathrm{d}^{n}}}{\mathrm{d}x^{n}}\varphi(x).
\]
It follows from \cite[Eq.~(26.8.37)]{Olver2010} that 
\[
g(\nabla)^{k}=\Delta^{k},
\]
where $\Delta$ is the difference operator defined by 
\[
(\Delta f)(x):=f(x+1)-f(x).
\]
It is easy to see that $g(\nabla)^{k}$ is a continuous operator on
$\mathcal{P}(\mathbb{R})$, hence its adjoint 
\[
\left(g(\nabla)^{k}\right)^{*}:\mathcal{P}'_{\pi_{\lambda,\beta}}(\mathbb{R})\longrightarrow\mathcal{P}'_{\pi_{\lambda,\beta}}(\mathbb{R})
\]
 is well-defined and is given by
\[
\langle\!\langle\varphi,\left(g(\nabla)^{k}\right)^{*}\Phi\rangle\!\rangle_{\pi_{\lambda,\beta}}=\langle\!\langle g(\nabla)^{k}\varphi,\Phi\rangle\!\rangle_{\pi_{\lambda,\beta}},\qquad\forall\Phi\in\mathcal{P}'_{\pi_{\lambda,\beta}}(\mathbb{R}),~\varphi\in\mathcal{P}(\mathbb{R}).
\]

As $\boldsymbol{1}\in L^{2}(\pi_{\lambda,\beta})\subset\mathcal{P}'_{\pi_{\lambda,\beta}}(\mathbb{R}),$
where $\boldsymbol{1}(x)=1,$ for all $x\in\mathbb{R},$ we define
the system of elements $Q_{k}^{\pi_{\lambda,\beta}}$, $k\in\mathbb{N}_{0}$
in $\mathcal{P}'_{\pi_{\lambda,\beta}}(\mathbb{R})$ by

\[
Q_{k}^{\pi_{\lambda,\beta}}:=\left(g(\nabla)^{k}\right)^{*}\boldsymbol{1},\quad k\in\mathbb{N}_{0}.
\]
The system $\mathbb{Q}^{\pi_{\lambda,\beta}}:=\big\{ Q_{k}^{\pi_{\lambda,\beta}}:=\left(g(\nabla)^{k}\right)^{*}\boldsymbol{1},k\in\mathbb{N}_{0}\big\}$
is called the \textit{generalized dual Appell system }\textit{\emph{associated
with $\pi_{\lambda,\beta}.$}} The pair $\mathbb{A}^{\pi_{\lambda,\beta}}:=(\mathbb{P}^{\pi_{\lambda,\beta}},\mathbb{Q}^{\pi_{\lambda,\beta}})$
is called \textit{generalized Appell system} generated by the measure
$\pi_{\lambda,\beta}$.
\begin{lem}
\label{lem:int-deriv-Cn}Let $n,k\in\mathbb{N}_{0}$ be given. Then
\[
\int_{\mathbb{R}}\nabla^{k}C_{n}^{\lambda,\beta}(x)\,\mathrm{d}\pi_{\lambda,\beta}(x)=\begin{cases}
k!s(n,k), & k\leq n,\\
0, & k>n,
\end{cases}
\]
where $s(n,k)$ is the Stirling numbers of the first kind.
\end{lem}

\begin{proof}
The case $k>n$ is clear since $\nabla^{k}C_{n}^{\lambda,\beta}(x)=0.$
For the case $k\leq n$, we first note that
\begin{align*}
\nabla^{k}C_{n}^{\lambda,\beta}(x) & =\frac{\mathrm{d}^{k}}{\mathrm{d}x^{k}}\frac{\mathrm{d}^{n}}{\mathrm{d}z^{n}}\frac{\mathrm{e}^{xf(z)}}{E_{\beta}(\lambda z)}\Bigg|_{z=0}=\frac{\mathrm{d}^{n}}{\mathrm{d}z^{n}}\left[f(z)^{k}\frac{\mathrm{e}^{xf(z)}}{E_{\beta}(\lambda z)}\right]\Bigg|_{z=0}\\
 & =\sum_{i=0}^{n}\binom{n}{i}\frac{\mathrm{d}^{n}}{\mathrm{d}z^{n}}k!\sum_{j=k}^{\infty}s(j,k)\frac{z^{j}}{j!}\Bigg|_{z=0}C_{n-i}^{\lambda,\beta}(x)\\
 & =k!\sum_{i=k}^{n}\binom{n}{i}s(i,k)C_{n-i}^{\lambda,\beta}(x).
\end{align*}
By Proposition \ref{prop:gen-Appell-prop}--(P5), we obtain
\begin{align*}
\int_{\mathbb{R}}\nabla^{k}C_{n}^{\lambda,\beta}(x)\,\mathrm{d}\pi_{\lambda,\beta}(x) & =k!\sum_{i=k}^{n}\binom{n}{i}s(i,k)\int_{\mathbb{R}}C_{n-i}^{\lambda,\beta}(x)\,\mathrm{d}\pi_{\lambda,\beta}(x)\\
 & =k!\sum_{i=k}^{n}\binom{n}{i}s(i,k)\delta_{n,i}\\
 & =k!s(n,k).\qedhere
\end{align*}
\end{proof}
\begin{thm}
\label{thm:biorthogonal-property}The generalized Appell polynomial
system $\mathbb{P}^{\pi_{\lambda,\beta}}$ and the generalized dual
Appell system $\mathbb{Q}^{\pi_{\lambda,\beta}}$ are biorthogonal
with respect to $L^{2}(\pi_{\lambda,\beta})$ and satisfies 
\[
\langle\!\langle C_{n}^{\lambda,\beta},Q_{m}^{\pi_{\lambda,\beta}}\rangle\!\rangle_{\pi_{\lambda,\beta}}=n!\delta_{n,m}.
\]
\end{thm}

\begin{proof}
It follows from Lemma \ref{lem:int-deriv-Cn} that 
\begin{align*}
\langle\!\langle C_{n}^{\lambda,\beta},Q_{m}^{\pi_{\lambda,\beta}}\rangle\!\rangle_{\pi_{\lambda,\beta}} & =\int_{\mathbb{R}}m!\sum_{k=m}^{\infty}\frac{S(k,m)}{k!}\left(\nabla^{k}C_{n}^{\lambda,\beta}(x)\right)\,\mathrm{d}\pi_{\lambda,\beta}(x)\\
 & =m!\sum_{k=m}^{\infty}\frac{S(k,m)}{k!}k!s(n,k)\\
 & =m!\sum_{k=m}^{n}s(n,k)S(k,m)\\
 & =n!\delta_{n,m}.
\end{align*}
The last equality holds due to, see \cite[Eq.~(28.8.39)]{Olver2010}
\[
\sum_{j=k}^{n}s(j,k)S(n,j)=\sum_{j=k}^{n}s(n,j)S(j,k)=\delta_{n,k},\quad\forall n,k\in\mathbb{N}_{0}.\qedhere
\]
\end{proof}
\begin{thm}
For every element $\Phi\in\mathcal{P}'_{\pi_{\lambda,\beta}}(\mathbb{R})$,
there exist a unique sequence $(\Phi_{k})_{k=0}^{\infty}\subset\mathbb{C}$
such that 
\begin{equation}
\Phi=\sum_{k=0}^{\infty}\Phi_{k}Q_{k}^{\pi_{\lambda,\beta}}.\label{eq:f3}
\end{equation}
Conversely, the entire series of the form \eqref{eq:f3} generates
a generalized function in $\mathcal{P}'_{\pi_{\lambda,\beta}}(\mathbb{R})$.
\end{thm}

\begin{proof}
Let $\Phi\in\mathcal{P}'_{\pi_{\lambda,\beta}}(\mathbb{R})$ be arbitrary.
For each $k\in\mathbb{N}_{0}$, let us consider the complex numbers
given by $\Phi_{k}:=\frac{1}{k!}\langle\!\langle C_{k}^{\lambda,\beta},\Phi\rangle\!\rangle_{\pi_{\lambda,\beta}}$,
and the functional in $\mathcal{P}(\mathbb{R})$ defined by 
\[
\mathcal{P}(\mathbb{R})\ni\varphi\mapsto{\displaystyle \sum_{k=0}^{\infty}k!\varphi_{k}\Phi_{k}\in\mathbb{C}.}
\]
Since this is a continuous linear functional which is given by the
inner product in $L^{2}(\pi_{\lambda,\beta})$, it defines an element
in $\mathcal{P}'_{\pi_{\lambda,\beta}}(\mathbb{R})$. We denote it
by $\Psi=\sum_{k=0}^{\infty}\Phi_{k}Q_{k}^{\pi_{\lambda,\beta}}$.
So, $\Psi$ is such that 
\[
\forall\varphi\in\mathcal{P}(\mathbb{R}),\ \langle\!\langle\varphi,\Psi\rangle\!\rangle_{\pi_{\lambda,\beta}}={\displaystyle \sum_{k=0}^{\infty}k!\varphi_{k}\Phi_{k}=\langle\!\langle\varphi,\Phi\rangle\!\rangle_{\pi_{\lambda,\beta}}.}
\]
Hence, it follows that $\Psi=\Phi$, since the representation for
$\Phi$ is unique.

Conversely, suppose that $\Phi=\sum_{k=0}^{\infty}\Phi_{k}Q_{k}^{\pi_{\lambda,\beta}}$,
$\Phi_{k}\in\mathbb{C}$, $k\in\mathbb{N}_{0}$. We are going to show
that $\Phi\in\mathcal{P}'_{\pi_{\lambda,\beta}}(\mathbb{R})$. Let
$\varphi\in\mathcal{P}(\mathbb{R})$ be of the form $\varphi=\sum_{k=0}^{n}\varphi_{k}C_{k}^{\lambda,\beta}.$
By Theorem \ref{thm:biorthogonal-property}, 
\[
\langle\!\langle\varphi,\Phi\rangle\!\rangle_{\pi_{\lambda,\beta}}={\displaystyle \sum_{k=0}^{n}k!\varphi_{k}\Phi_{k}=(\!(\varphi,\bar{\Phi})\!)_{L^{2}(\pi_{\lambda,\beta})}}.
\]
This is clearly a linear map and is also continuous in the topology
of $\mathcal{P}(\mathbb{R})$, being given by the inner product in
$L^{2}(\pi_{\lambda,\beta})$. These then defines $\Phi$ an element
in $\mathcal{P}'_{\pi_{\lambda,\beta}}(\mathbb{R})$.
\end{proof}

\section{Test and Generalized Function Spaces}

\label{sec:Test-Generalized-Function-Spaces}Given $\kappa\in[0,1]$
and $q\in\mathbb{N}_{0}$, let $\varphi\in\mathcal{P}(\mathbb{R})$
be such that $\varphi=\sum_{n=0}^{k}\varphi_{n}C_{n}^{\lambda,\beta}$.
We introduce in $\mathcal{P}(\mathbb{R})$ a Hilbert norm by 
\[
\|\varphi\|_{q,\kappa,\pi_{\lambda,\beta}}^{2}:=\sum_{n=0}^{k}(n!)^{1+\kappa}2^{nq}|\varphi_{n}|^{2}.
\]
The completion of $\mathcal{P}(\mathbb{R})$ in the norm $\|\cdot\|_{q,\kappa,\pi_{\lambda,\beta}}$
is denoted by $(H)_{q,\pi_{\lambda,\beta}}^{\kappa}$, so $\mathcal{P}(\mathbb{R})\hookrightarrow(H)_{q,\pi_{\lambda,\beta}}^{\kappa}$
densely. The space $(H)_{q,\pi_{\lambda,\beta}}^{\kappa}$ is a Hilbert
space with inner product given by 
\[
(\!(\varphi,\psi)\!)_{q,\pi_{\lambda,\beta}}^{\kappa}:=\sum_{n=0}^{\infty}(n!)^{1+\kappa}2^{nq}\varphi_{n}\bar{\psi}_{n},
\]
admitting the representation
\[
(H)_{q,\pi_{\lambda,\beta}}^{\kappa}:=\left\{ \varphi:\mathbb{R}\rightarrow\mathbb{C}\,\middle|\,\varphi=\sum_{n=0}^{\infty}\varphi_{n}C_{n}^{\lambda,\beta},\;\|\varphi\|_{q,\kappa,\pi_{\lambda,\beta}}^{2}=\sum_{n=0}^{\infty}(n!)^{1+\kappa}2^{nq}|\varphi_{n}|^{2}<\infty\right\} .
\]
 We also have that the inclusion $(H)_{q,\pi_{\lambda,\beta}}^{\kappa}\subset L^{2}(\pi_{\lambda,\beta})$
is dense which results from Remark \ref{rem:H2} on $\pi_{\lambda,\beta}$.
In this way we obtain the following dense chain of Hilbert spaces:
\begin{equation}
\dots\subset(H)_{q+1,\pi_{\lambda,\beta}}^{\kappa}\subset(H)_{q,\pi_{\lambda,\beta}}^{\kappa}\subset\dots\subset(H)_{0,\pi_{\lambda,\beta}}^{\kappa}\subset L^{2}(\pi_{\lambda,\beta}).\label{eq:c4}
\end{equation}
For $p>q$, the injection operator $I_{p,q}:(H)_{p,\pi_{\lambda,\beta}}^{\kappa}\longrightarrow(H)_{q,\pi_{\lambda,\beta}}^{\kappa}$
is Hilbert-Schmidt. In fact, the set $\left\{ \bar{C}_{n}^{\lambda,\beta}:=\left((n!)^{1+\kappa}2^{np}\right)^{-\frac{1}{2}}C_{n}^{\lambda,\beta}\,|\,n\in\mathbb{N}_{0}\right\} $
is a total orthonormal set in $(H)_{p,\pi_{\lambda,\beta}}^{\kappa}$.
Then the Hilbert-Schmidt norm of $I_{p,q}$ is given by 
\begin{align*}
\|I_{p,q}\|_{HS}^{2}=\sum_{n=0}^{\infty}\|I_{p,q}\bar{C}_{n}^{\lambda,\beta}\|_{q,\kappa,\pi_{\lambda,\beta}}^{2} & =\sum_{n=0}^{\infty}(n!)^{1+\kappa}2^{nq}\left((n!)^{1+\kappa}2^{np}\right)^{-1}=\sum_{n=0}^{\infty}\left(\frac{1}{2^{p-q}}\right)^{n}<\infty.
\end{align*}
Given $\kappa\in[0,1]$ the test function space associated with $\pi_{\lambda,\beta}$
is defined by 
\[
(N)_{\pi_{\lambda,\beta}}^{\kappa}:=\bigcap_{q=0}^{\infty}(H)_{q,\pi_{\lambda,\beta}}^{\kappa},
\]
which is a nuclear space.
\begin{example}
\label{exa:wick-exp-as-test-func}Let us consider the Wick exponential
$\mathrm{e}_{\pi_{\lambda,\beta}}(f(z),\cdot)$, $z\in\mathbb{C}$.
For $q\in\mathbb{N}_{0}$, 
\begin{align*}
\|\mathrm{e}_{\pi_{\lambda,\beta}}(f(z),\cdot)\|_{q,\kappa,\pi_{\lambda,\beta}}^{2} & =\sum_{n=0}^{\infty}(n!)^{1+\kappa}2^{nq}\frac{|z|^{2n}}{(n!)^{2}}=\sum_{n=0}^{\infty}\frac{1}{2^{n\kappa}}\frac{\left(2^{\kappa}2^{q}|z|^{2}\right)^{n}}{(n!)^{1-\kappa}}.
\end{align*}
If $\kappa=0,$ we have 
\[
\|\mathrm{e}_{\pi_{\lambda,\beta}}(f(z),\cdot)\|_{q,0,\pi_{\lambda,\beta}}^{2}=\exp(2^{q}|z|^{2})<\infty,\,\forall z\in\mathbb{C}.
\]
For $\kappa\in(0,1)$, we use H{\"o}lder's inequality with the pair
$(\frac{1}{\kappa},\frac{1}{1-\kappa})$ and obtain 
\begin{align*}
\|\mathrm{e}_{\pi_{\lambda,\beta}}(f(z),\cdot)\|_{q,\kappa,\pi_{\lambda,\beta}}^{2} & \leq\left(\sum_{n=0}^{\infty}\left(\frac{1}{2^{n\kappa}}\right)^{\frac{1}{\kappa}}\right)^{\kappa}\left(\sum_{n=0}^{\infty}\left(\frac{\left(2^{\kappa}2^{q}|z|^{2}\right)^{n}}{(n!)^{1-\kappa}}\right)^{\frac{1}{1-\kappa}}\right)^{1-\kappa}\\
 & =2^{\kappa}\exp\left((1-\kappa)2^{\frac{\kappa+q}{1-\kappa}}|z|^{\frac{2}{1-\kappa}}\right)<\infty,
\end{align*}
for all $z\in\mathbb{C}.$ Hence $\mathrm{e}_{\pi_{\lambda,\beta}}(f(z),\cdot)\in(N)_{\pi_{\lambda,\beta}}^{\kappa}$,
$\kappa\in[0,1)$. For $\kappa=1$ and $q\in\mathbb{N}_{0},$ we have
\[
\|\mathrm{e}_{\pi_{\lambda,\beta}}(f(z),\cdot)\|_{q,1,\pi_{\lambda,\beta}}^{2}=\sum_{n=0}^{\infty}(2^{q}|z|^{2})^{n}
\]
which is convergent if and only if $|z|<2^{-q/2}.$ Thus, $\mathrm{e}_{\pi_{\lambda,\beta}}(f(z),\cdot)\notin(N)_{\pi_{\lambda,\beta}}^{1}$,
$z\in\mathbb{C\setminus}\{0\}$. However, for each $q\in\mathbb{N}_{0},$
$\mathrm{e}_{\pi_{\lambda,\beta}}(f(z),\cdot)\in(H)_{q,\pi_{\lambda,\beta}}^{1}$
provided that $|z|<2^{-q/2}.$
\end{example}

\begin{prop}
\label{prop:continuous-extension}
\begin{enumerate}
\item Every test functions $\varphi\in(N)_{\pi_{\lambda,\beta}}^{1}$ has
a unique extension to the set $\mathbb{C}$ such that $\varphi$ is
an entire function of minimal type and order of growth one, that is,
$\varphi:\mathbb{C}\rightarrow\mathbb{C}$ is entire and for all $\varepsilon>0$
there is $K>0$ such that $|\varphi(z)|\leq K\mathrm{e}^{\varepsilon|z|}.$
\item For $\varepsilon>0$, there exists $\sigma_{\varepsilon}>0$ in which
$2^{q}>\sigma_{\varepsilon}^{-2}$, for $q\in\mathbb{N}_{0}$ and
$K'>0$ such that we have the following bound 
\[
|\varphi(z)|\leq K'\|\varphi\|_{q,1,\pi_{\lambda,\beta}}\mathrm{e}^{\varepsilon|z|},\quad z\in\mathbb{C}.
\]
\end{enumerate}
\end{prop}

\begin{proof}
Let $\varphi\in(N)_{\pi_{\lambda,\beta}}^{1}$ be given. Then $\varphi$
can be expressed as
\[
\varphi=\sum_{n=0}^{\infty}\varphi_{n}C_{n}^{\lambda,\beta},\quad\varphi_{n}\in\mathbb{C},\;n\in\mathbb{N}_{0},
\]
with 
\[
\|\varphi\|_{q,1,\pi_{\lambda,\beta}}^{2}=\sum_{n=0}^{\infty}(n!)^{2}2^{nq}|\varphi_{n}|^{2}
\]
for all $q\in\mathbb{N}_{0}.$ Let $\varepsilon>0$ and $z\in\mathbb{C}$.
By Proposition \ref{prop:gen-Appell-prop}-(P6), there exist $C_{\varepsilon},\sigma_{\varepsilon}>0$
such that
\begin{align*}
|\varphi(z)|\leq\sum_{n=0}^{\infty}|\varphi_{n}C_{n}^{\lambda,\beta}(z)| & \leq\sum_{n=0}^{\infty}|\varphi_{n}||C_{n}^{\lambda,\beta}(z)|\leq\mathrm{e}^{\varepsilon|z|}C_{\varepsilon}\sum_{n=0}^{\infty}n!|\varphi_{n}|\sigma_{\varepsilon}^{-n}.
\end{align*}
Using H{\"o}lder's inequality, we have 
\begin{align}
|\varphi(z)| & \leq C_{\varepsilon}\mathrm{e}^{\varepsilon|z|}\left(\sum_{n=0}^{\infty}(n!)^{2}2^{nq}|\varphi_{n}|^{2}\right)^{1/2}\left(\sum_{n=0}^{\infty}2^{-nq}\sigma_{\varepsilon}^{-2n}\right)^{1/2}\nonumber \\
 & =C_{\varepsilon}(1-2^{-q}\sigma_{\varepsilon}^{-2})^{-1/2}\|\varphi\|_{q,1,\pi_{\lambda,\beta}}\mathrm{e}^{\varepsilon|z|},\label{eq:P6-holder-inequality}
\end{align}
for $2^{q}>\sigma_{\varepsilon}^{-2}$. By the Weierstrass M-test,
the series $\varphi(z)=\sum_{n=0}^{\infty}\varphi_{n}C_{n}^{\lambda,\beta}(z)$
converges uniformly and absolutely in any neighborhood of zero in
$\mathbb{C}.$ Since each term $\varphi_{n}C_{n}^{\lambda,\beta}(z)$
is entire in $z$, the uniform convergence implies that $z\mapsto\sum_{n=0}^{\infty}\varphi_{n}C_{n}^{\lambda,\beta}(z)$
is entire on $\mathbb{C}$. On the other hand, take $K'=C_{\varepsilon}(1-2^{-q}\sigma_{\varepsilon}^{-2})^{-1/2}$
in equation \eqref{eq:P6-holder-inequality} and so we obtain the
required bound
\[
|\varphi(z)|\leq K'\|\varphi\|_{q,1,\pi_{\lambda,\beta}}\mathrm{e}^{\varepsilon|z|},\quad z\in\mathbb{C}.
\]
This completes the proof.
\end{proof}

As the inclusion $(H)_{q,\pi_{\lambda,\beta}}^{\kappa}\hookrightarrow L^{2}(\pi_{\lambda,\beta})$
is dense, we may compute the dual of $(H)_{q,\pi_{\lambda,\beta}}^{\kappa}$
with respect to $L^{2}(\pi_{\lambda,\beta})$, that is, the functionals
are represented in terms of the inner product $(\!(\cdot,\cdot)\!)_{\pi_{\lambda,\beta}}$.
In the literature this process is known as rigged Hilbert spaces,
see for example \cite{BK88}. We are not going to reproduce this process
here. The resulting triplet of Hilbert spaces is 
\[
(H)_{q,\pi_{\lambda,\beta}}^{\kappa}\subset L^{2}(\pi_{\lambda,\beta})\subset(H)_{-q,\pi_{\lambda,\beta}}^{-\kappa}.
\]

The Hilbert space $(H)_{-q,\pi_{\lambda,\beta}}^{-\kappa}$ which
corresponds to the completion of $L^{2}(\pi_{\lambda,\beta})$ with
respect to the norm $\|\cdot\|_{-q,-\kappa,\pi_{\lambda,\beta}}$,
admit the following representation
\[
(H)_{-q,\pi_{\lambda,\beta}}^{-\kappa}=\left\{ \Phi=\sum_{n=0}^{\infty}\Phi_{n}Q_{n}^{\pi_{\lambda,\beta}},~\Phi_{n}\in\mathbb{C}\,\middle|\,\|\Phi\|_{-q,-\kappa,\pi_{\lambda,\beta}}^{2}=\sum_{n=0}^{\infty}(n!)^{1-\kappa}2^{-nq}|\Phi_{n}|^{2}<\infty\right\} .
\]
From the general theory of duality (see for example \cite{S71}) the
dual of $(N)_{\pi_{\lambda,\beta}}^{\kappa}$ with respect to $L^{2}(\pi_{\lambda,\beta})$
is given by 
\[
(N)_{\pi_{\lambda,\beta}}^{-\kappa}=\bigcup_{q=0}^{\infty}(H)_{-q,\pi_{\lambda,\beta}}^{-\kappa}.
\]
As a result we obtain the chain of continuous embeddings 
\[
(N)_{\pi_{\lambda,\beta}}^{\kappa}\subset\dots\subset(H)_{q,\pi_{\lambda,\beta}}^{\kappa}\subset\dots\subset L^{2}(\pi_{\lambda,\beta})\subset\dots\subset(H)_{-q,\pi_{\lambda,\beta}}^{-\kappa}\subset\dots\subset(\!N)_{\pi_{\lambda,\beta}}^{-\kappa}.
\]

\begin{example}[Generalized Radon-Nikodym derivative]
\label{exa:GRN-derivative}We want to define the generalized function
$\rho_{\pi_{\lambda,\beta}}(z,\cdot)\in(N)_{\pi_{\lambda,\beta}}^{-1},$
$z\in\mathbb{C}$, such that 
\begin{equation}
\langle\!\langle\varphi,\rho_{\pi_{\lambda,\beta}}(z,\cdot)\rangle\!\rangle_{\pi_{\lambda,\beta}}=\int_{\mathbb{R}}\varphi(x-z)\,\mathrm{d}\pi_{\lambda,\beta}(x),\quad\varphi\in(N)_{\pi_{\lambda,\beta}}^{1}.\label{5}
\end{equation}
Taking into account the result of Proposition \ref{prop:continuous-extension}--(2),
it turns out that 
\[
|\langle\!\langle\varphi,\rho_{\pi_{\lambda,\beta}}(z,\cdot)\rangle\!\rangle_{\pi_{\lambda,\beta}}|\leq K'\|\varphi\|_{q,1,\pi_{\lambda,\beta}}\mathrm{e}^{\varepsilon|z|}\int_{\mathbb{R}}\mathrm{e}^{\varepsilon|x|}\,\mathrm{d}\pi_{\lambda,\beta}(x).
\]
The integral on the right hand side is finite by Proposition \ref{prop:Analytic-LT-fPm}.
Thus, we have in fact $\rho_{\pi_{\lambda,\beta}}(z,\cdot)\in(N)_{\pi_{\lambda,\beta}}^{-1}.$
By Proposition \ref{prop:gen-Appell-prop}--(P4) and (P5), we have
\begin{align*}
\langle\!\langle C_{n}^{\lambda,\beta},\rho_{\pi_{\lambda,\beta}}(z,\cdot)\rangle\!\rangle_{\pi_{\lambda,\beta}} & =\int_{\mathbb{R}}C_{n}^{\lambda,\beta}(x-z)\,\mathrm{d}\pi_{\lambda,\beta}(x)\\
 & ={\displaystyle \sum_{k=0}^{n}\binom{n}{k}(-z)_{n-k}\int_{\mathbb{R}}}C_{k}^{\lambda,\beta}(x)\,\mathrm{d}\pi_{\lambda,\beta}(x)\\
 & =\sum_{k=0}^{n}\binom{n}{k}(-z)_{n-k}\delta_{k,0}\\
 & =(-z)_{n}\\
 & =\left\langle \!\!\!\left\langle C_{n}^{\lambda,\beta},\sum_{k=0}^{\infty}\frac{(-z)_{k}}{k!}Q_{k}^{\pi_{\lambda,\beta}}\right\rangle \!\!\!\right\rangle _{\pi_{\lambda,\beta}}
\end{align*}
where we used the biorthogonal property of $\mathbb{P}^{\pi_{\lambda,\beta}}$
and $\mathbb{Q}^{\pi_{\lambda,\beta}}$. This implies that $\rho_{\pi_{\lambda,\beta}}(-z,\cdot)$
has the following representation 
\begin{equation}
\rho_{\pi_{\lambda,\beta}}(-z,\cdot)=\sum_{k=0}^{\infty}\frac{(z)_{k}}{k!}Q_{k}^{\pi_{\lambda,\beta}}(\cdot).\label{6}
\end{equation}
In other words, $\rho_{\pi_{\lambda,\beta}}(-z,\cdot)$ is the generating
generalized function of the system $\mathbb{Q}^{\pi_{\lambda,\beta}}$.
\end{example}

\begin{example}[Delta distribution]
 For $z\in\mathbb{C},$ we define the distribution $\delta_{z}$
by the following $\mathbb{Q}^{\pi_{\lambda,\beta}}$-decomposition:
\[
\delta_{z}=\sum_{n=0}^{\infty}\frac{C_{n}^{\lambda,\beta}(z)}{n!}Q_{n}^{\pi_{\lambda,\beta}}.
\]
By Proposition \ref{prop:gen-Appell-prop}-(P6), given $\varepsilon>0,$
there exist $C_{\varepsilon},\sigma_{\varepsilon}>0$ such that 
\begin{align*}
\|\delta_{z}\|_{-q,-\kappa,\pi_{\lambda,\beta}}^{2} & =\sum_{n=0}^{\infty}(n!)^{-1-\kappa}2^{-nq}|C_{n}^{\lambda,\beta}(z)|^{2}\\
 & \leq C_{\varepsilon}^{2}\mathrm{e}^{2\varepsilon|z|}\sum_{n=0}^{\infty}2^{-nq}\sigma_{\varepsilon}^{-2n}(n!)^{1-\kappa}\\
 & \leq C_{\varepsilon}^{2}\mathrm{e}^{2\varepsilon|z|}\left(\sum_{n=0}^{\infty}\left(2^{-nq}\right)^{\frac{1}{2-\kappa}}\right)^{2-\kappa}\left(\sum_{n=0}^{\infty}\left(\frac{\sigma_{\varepsilon}^{-2n}}{(n!)^{\kappa-1}}\right)^{\frac{1}{\kappa-1}}\right)^{\kappa-1}\\
 & \leq C_{\varepsilon}^{2}\mathrm{e}^{2\varepsilon|z|}\left(\sum_{n=0}^{\infty}\left(\frac{1}{2^{q/2-\kappa}}\right)^{n}\right)^{2-\kappa}\exp\left((\kappa-1)\sigma_{\varepsilon}^{2/1-\kappa}\right)
\end{align*}
which is finite for $q\geq2-\kappa$, $\kappa\in[0,1)$. Hence, $\delta_{z}\in(H)_{-q,\pi_{\lambda,\beta}}^{-\kappa}\subset(N)_{\pi_{\lambda,\beta}}^{-\kappa}$,
$\kappa\in[0,1)$. Also, for $\kappa=1$, we have 
\begin{align*}
\|\delta_{z}\|_{-q,-1,\pi_{\lambda,\beta}}^{2} & =\sum_{n=0}^{\infty}(n!)^{-2}2^{-nq}|C_{n}^{\lambda,\beta}(z)|^{2}\\
 & \leq C_{\varepsilon}^{2}\mathrm{e}^{2\varepsilon|z|}\sum_{n=0}^{\infty}2^{-nq}\sigma_{\varepsilon}^{-2n},
\end{align*}
which is finite for sufficiently large $q\in\mathbb{N}.$ Thus, $\delta_{z}\in(N)_{\pi_{\lambda,\beta}}^{-1}.$
For $\varphi=\sum_{n=0}^{\infty}\varphi_{n}C_{n}^{\lambda,\beta}$,
the action of $\delta_{z}$ is given by 
\[
\langle\!\langle\delta_{z},\varphi\rangle\!\rangle_{\pi_{\lambda,\beta}}=\left\langle \!\!\!\left\langle \sum_{m=0}^{\infty}\varphi_{m}C_{m}^{\lambda,\beta},\sum_{n=0}^{\infty}\frac{C_{n}^{\lambda,\beta}(z)}{n!}Q_{n}^{\pi_{\lambda,\beta}}\right\rangle \!\!\!\right\rangle _{\pi_{\lambda,\beta}}=\sum_{n=0}^{\infty}\varphi_{n}C_{n}^{\lambda,\beta}(z)=\varphi(z)
\]
by the biorthogonal property of $\mathbb{P}^{\pi_{\lambda,\beta}}$
and $\mathbb{Q}^{\pi_{\lambda,\beta}}$. This implies that $\delta_{z}$
(in particular for $z$ real) plays the role of a ``$\delta$-function''
(evaluation map) in calculus.
\end{example}

\section{Characterization Theorems}

\label{sec:Characterization-Theorems}In this section, we define two
integrals transforms, called $S_{\pi_{\lambda,\beta}}$--transform
and convolution $C_{\pi_{\lambda,\beta}}$, which are used to characterize
the test function spaces $(N)_{\pi_{\lambda,\beta}}^{\kappa}$ and
generalized function spaces $(N)_{\pi_{\lambda,\beta}}^{-\kappa}$.
The \emph{$S_{\pi_{\lambda,\beta}}$--transform} of $\varphi\in(N)_{\pi_{\lambda,\beta}}^{\kappa}$
is defined by 
\[
(S_{\pi_{\lambda,\beta}}\varphi)(z):=\int_{\mathbb{R}}\varphi(x)\mathrm{e}_{\pi_{\lambda,\beta}}(z,x)\,\mathrm{d}\pi_{\lambda,\beta}(x).
\]
The $S_{\pi_{\lambda,\beta}}$--transform may be extended to $\Phi\in(N)_{\pi_{\lambda,\beta}}^{-\kappa}$
(in a consistent way) as follows
\[
(S_{\pi_{\lambda,\beta}}\Phi)(z):=\langle\!\langle\mathrm{e}_{\pi_{\lambda,\beta}}(z;\cdot),\Phi\rangle\!\rangle_{\pi_{\lambda,\beta}}.
\]
Note that for $\kappa=1$, $S_{\pi_{\lambda,\beta}}\Phi$ is defined
only in a neighborhood of the zero in $\mathbb{C}$, because if $\Phi\in(H)_{-q,\pi_{\lambda,\beta}}^{-1}$,
then $\mathrm{e}_{\pi_{\lambda,\beta}}(z;\cdot)\in(H)_{q,\pi_{\lambda,\beta}}^{1}$
for $z\in\mathbb{C}$ such that $|z|<2^{-q/2}$ as shown in Example~\ref{exa:wick-exp-as-test-func}.
Now, we introduce the second integral transform which is more appropriate
to characterize the test function spaces $(N)_{\pi_{\lambda,\beta}}^{\kappa}$.
The \emph{convolution} $C_{\pi_{\lambda,\beta}}$ of $\varphi\in(N)_{\pi_{\lambda,\beta}}^{\kappa}$
is defined by 
\[
(C_{\pi_{\lambda,\beta}}\varphi)(z):=\int_{\mathbb{R}}\varphi(x+z)\,\mathrm{d}\pi_{\lambda,\beta}(x)
\]
and using Example~\ref{exa:GRN-derivative} may be written as
\[
(C_{\pi_{\lambda,\beta}}\varphi)(z)=\langle\!\langle\varphi,\rho_{\pi_{\lambda,\beta}}(-z,\cdot)\rangle\!\rangle_{\pi_{\lambda,\beta}}.
\]
In Gaussian analysis, the convolution $C_{\pi_{\lambda,\beta}}$ and
the $S_{\pi_{\lambda,\beta}}$--transform coincide, however these
two transformations are different in the fractional Poisson analysis,
or more generally, in non-Gaussian analysis.

\subsection{Characterization of Test Functions}

For every $l\in\mathbb{N}_{0}$, we denote by $\mathcal{E}_{2^{-l}}^{k}(\mathbb{C})$
the set of entire functions of order of growth $k\in[1,2]$ and type
$2^{-l}$, i.e., 
\[
\mathcal{E}_{2^{-l}}^{k}(\mathbb{C})=\big\{ u:\mathbb{C}\rightarrow\mathbb{C}~\text{is entire}\mid|u(z)|\leq C\exp(2^{-l}|z|^{k}),C>0\big\}.
\]
For any $l\in\mathbb{N}_{0},$ the map 
\[
|\cdot|_{l,k}:\mathcal{E}_{2^{-l}}^{k}(\mathbb{C})\rightarrow\mathbb{R}_{+},~u\mapsto|u|_{l,k}:=\sup_{z\in\mathbb{C}}\big\{|u(z)|\exp(-2^{-l}|z|^{k})\big\},
\]
is a norm in $\mathcal{E}_{2^{-l}}^{k}(\mathbb{C})$ and $(\mathcal{E}_{2^{-l}}^{k}(\mathbb{C}),|\cdot|_{l,k})$
is a Banach space. The space of entire functions of minimal type and
order of growth $k$ is defined by 
\[
\mathcal{E}_{\min}^{k}(\mathbb{C}):=\bigcap_{l\in\mathbb{N}_{0}}\mathcal{E}_{2^{-l}}^{k}(\mathbb{C}).
\]
Any entire function $u$ can be represented in Taylor series in the
form 
\[
u(z)=\sum_{n=0}^{\infty}u_{n}z^{n},~~~u_{n}=\frac{1}{n!}\frac{d^{n}}{dz^{n}}u(z)\bigg|_{z=0},\quad z\in\mathbb{C}.
\]
Let us consider the family $(\vertiii{\cdot}{}_{q,\kappa})_{q\in\mathbb{N}_{0}},$
$\kappa\in[0,1]$ of Hilbert norms in $\mathcal{E}_{\min}^{k}(\mathbb{C})$,
defined by 
\[
\vertiii{u}_{q,\kappa}^{2}:=\sum_{n=0}^{\infty}(n!)^{1+\kappa}2^{nq}|u_{n}|^{2}<\infty,\quad u\in\mathcal{E}_{\min}^{k}(\mathbb{C}).
\]
This family of Hilbert norms is equivalent to the family of norms
$(|\cdot|{}_{l,k})_{l\in\mathbb{N}_{0}},$ see Appendix \ref{sec:Equivalent-norms}.
\begin{thm}
\label{thm:Convolution-N1-E1min}The convolution $C_{\pi_{\lambda,\beta}}$
is a homeomorphism between the test function space $(N)_{\pi_{\lambda,\beta}}^{1}$
and the space $\mathcal{E}_{\min}^{1}(\mathbb{C})$ of entire functions
of minimal type and order of growth one.
\end{thm}

\begin{proof}
Let $\varphi\in(N)_{\pi_{\lambda,\beta}}^{1}$ be of the form $\varphi=\sum_{n=0}^{\infty}\varphi_{n}C_{n}^{\lambda,\beta}$
with 
\[
\|\varphi\|_{q,1,\pi_{\lambda,\beta}}^{2}=\sum_{n=0}^{\infty}(n!)^{2}2^{nq}|\varphi_{n}|^{2}<\infty
\]
for each $q\in\mathbb{N}_{0}$. So we have 
\[
|\varphi_{n}|\leq(n!)^{-1}2^{-nq/2}\|\varphi\|_{q,1,\pi_{\lambda,\beta}}.
\]
It follows from Example~\ref{exa:GRN-derivative} that
\[
(C_{\pi_{\lambda,\beta}}\varphi)(z)=\sum_{n=0}^{\infty}\varphi_{n}(z)_{n}=\sum_{n=0}^{\infty}\varphi_{n}\sum_{k=0}^{n}s(n,k)z^{k}=\sum_{k=0}^{\infty}\sum_{n=0}^{\infty}\varphi_{n}s(n,k)z^{k},
\]
where in the second equality we use Proposition \ref{prop:Stirling-falling-factorial}--1.
From Proposition~\ref{prop:Stirling-numbers-explicit}--1, we have
\[
s(n,k)=\frac{A_{n}^{k}}{k!}.
\]
For $\varepsilon>0$, we define 
\[
F_{\varepsilon}:=\sup_{|z|=\varepsilon}|f(z)|,
\]
where $f(z)=\log(1+z)$, so that by Cauchy inequality
\begin{align*}
|s(n,k)| & \le\frac{1}{k!}\sum_{l_{1}+\dots+l_{k}=n}\frac{n!}{l_{1}!\dots l_{k}!}|f^{(l_{1})}(0)|\dots|f^{(l_{k})}(0)|\\
 & \leq\frac{1}{k!}\sum_{l_{1}+\dots+l_{k}=n}\frac{n!l_{1}!\dots l_{k}!}{l_{1}!\dots l_{k}!}F_{\varepsilon}^{k}\varepsilon^{-n}\\
 & \leq\frac{n!}{k!}F_{\varepsilon}^{k}2^{n}\varepsilon^{-n}.
\end{align*}
Now, we estimate $\sum_{n=0}^{\infty}\varphi_{n}s(n,k)$ as follows
\begin{align*}
\Bigg|\sum_{n=0}^{\infty}\varphi_{n}s(n,k)\Bigg| & \leq\sum_{n=0}^{\infty}|\varphi_{n}||s(n,k)|\\
 & \leq\sum_{n=0}^{\infty}(n!)^{-1}2^{-nq/2}\|\varphi\|_{q,1,\pi_{\lambda,\beta}}\frac{n!}{k!}F_{\varepsilon}^{k}2^{n}\varepsilon^{-n}\\
 & =\frac{\|\varphi\|_{q,1,\pi_{\lambda,\beta}}F_{\varepsilon}^{k}}{k!}\sum_{n=0}^{\infty}\left(2^{\frac{2-q}{2}}\varepsilon^{-1}\right)^{n}\\
 & =\frac{\|\varphi\|_{q,1,\pi_{\lambda,\beta}}F_{\varepsilon}^{k}}{k!}(1-2^{\frac{2-q}{2}}\varepsilon^{-1})^{-1}
\end{align*}
for $2^{\frac{2-q}{2}}\varepsilon^{-1}<1$. It follows that
\begin{align*}
\vertiii{(C_{\pi_{\lambda,\beta}}\varphi)(z)}_{q,1}^{2} & \leq\sum_{k=0}^{\infty}(k!)^{2}2^{kq}\frac{\|\varphi\|_{q,1,\pi_{\lambda,\beta}}^{2}F_{\varepsilon}^{2k}}{(k!)^{2}}(1-2^{\frac{2-q}{2}}\varepsilon^{-1})^{-2}\\
 & =\|\varphi\|_{q,1,\pi_{\lambda,\beta}}^{2}(1-2^{\frac{2-q}{2}}\varepsilon^{-1})^{-2}\sum_{k=0}^{\infty}(2^{q}F_{\varepsilon}^{2})^{k}\\
 & =\|\varphi\|_{q,1,\pi_{\lambda,\beta}}^{2}(1-2^{\frac{2-q}{2}}\varepsilon^{-1})^{-2}(1-2^{q}F_{\varepsilon}^{2})^{-1},
\end{align*}
where $\varepsilon>0$ is such that $2^{q}F_{\varepsilon}^{2}<1$.
So $C_{\pi_{\lambda,\beta}}$ is continuous.

Conversely, let $u\in\mathcal{E}_{\min}^{k}(\mathbb{C})$ be given
such that 
\begin{align*}
u(z) & =\sum_{n=0}^{\infty}u_{n}z^{n}
\end{align*}
with
\[
\vertiii{u}_{q,1}^{2}=\sum_{n=0}^{\infty}(n!)^{2}2^{nq}|u_{n}|^{2}<\infty
\]
for each $q\in\mathbb{N}_{0}$. Hence, 
\[
|u_{n}|\leq(n!)^{-1}2^{-nq/2}\vertiii{u}_{q,1}.
\]
Note that 
\[
u(z)=\sum_{n=0}^{\infty}u_{n}z^{n}=\sum_{n=0}^{\infty}u_{n}\sum_{k=0}^{n}S(n,k)(z)_{k}=\sum_{k=0}^{\infty}\sum_{n=0}^{\infty}u_{n}S(n,k)(z)_{k}
\]
by Proposition \ref{prop:Stirling-falling-factorial}--2. From Proposition~\ref{prop:Stirling-numbers-explicit}--2,
we have 
\[
S(n,k)=\frac{B_{n}^{k}}{k!}.
\]
For $\varepsilon>0$, we define 
\[
G_{\varepsilon}:=\sup_{|z|=\varepsilon}|g(z)|
\]
where $g(z)=\mathrm{e}^{z}-1$ such that by Cauchy inequality
\begin{align*}
|S(n,k)| & \leq\frac{1}{k!}\sum_{l_{1}+\dots+l_{k}=n}\frac{n!}{l_{1}!\dots l_{k}!}|g^{(l_{1})}(0)|\dots|g^{(l_{k})}(0)|\\
 & \leq\frac{1}{k!}\sum_{l_{1}+\dots+l_{k}=n}\frac{n!l_{1}!\dots l_{k}!}{l_{1}!\dots l_{k}!}G_{\varepsilon}^{k}\varepsilon^{-n}\\
 & \leq\frac{n!}{k!}G_{\varepsilon}^{k}2^{n}\varepsilon^{-n}.
\end{align*}
Now, we estimate $\sum_{n=0}^{\infty}u_{n}S(n,k)$ as follows 
\begin{align*}
\Bigg|\sum_{n=0}^{\infty}u_{n}S(n,k)\Bigg| & \leq\sum_{n=0}^{\infty}|u_{n}|S(n,k)\\
 & \leq\sum_{n=0}^{\infty}(n!)^{-1}2^{-nq/2}\vertiii{u}_{q,1}\frac{n!}{k!}G_{\varepsilon}^{k}2^{n}\varepsilon^{-n}\\
 & \leq\frac{\vertiii{u}_{q,1}G_{\varepsilon}^{k}}{k!}\sum_{n=0}^{\infty}\left(2^{\frac{2-q}{2}}\varepsilon^{-1}\right)^{n}\\
 & =\frac{\vertiii{u}_{q,1}G_{\varepsilon}^{k}}{k!}(1-2^{\frac{2-q}{2}}\varepsilon^{-1})^{-1}
\end{align*}
for $2^{\frac{2-q}{2}}\varepsilon^{-1}<1$. Define $\psi$ of the
form 
\[
\psi=\sum_{k=0}^{\infty}\psi_{k}C_{k}^{\lambda,\beta}
\]
where $\psi_{k}=\sum_{n=0}^{\infty}u_{n}S(n,k)$. Then we have 
\begin{align}
\|\psi\|_{q,1,\pi_{\lambda,\beta}}^{2} & \leq\sum_{k=0}^{\infty}(k!)^{2}2^{kq}\frac{\vertiii{u}_{q,1}^{2}G_{\varepsilon}^{2k}}{(k!)^{2}}(1-2^{\frac{2-q}{2}}\varepsilon^{-1})^{-2}\nonumber \\
 & =\vertiii{u}_{q,1}^{2}(1-2^{\frac{2-q}{2}}\varepsilon^{-1})^{-2}\sum_{k=0}^{\infty}(2^{q}G_{\varepsilon}^{2})^{k}\nonumber \\
 & =\vertiii{u}_{q,1}^{2}(1-2^{\frac{2-q}{2}}\varepsilon^{-1})^{-2}(1-2^{q}G_{\varepsilon}^{2})^{-1},\label{eq:psi-u-inequality}
\end{align}
where $\varepsilon>0$ is such that $2^{q}G_{\varepsilon}^{2}<1$.
Hence, $\psi\in(N)_{\pi_{\lambda,\beta}}^{1}$ and $C_{\pi_{\lambda,\beta}}\psi=u$.
Moreover, if $C_{\pi_{\lambda,\beta}}\psi=0$, equation \eqref{eq:psi-u-inequality}
implies that $\psi=0$ and so $C_{\pi_{\lambda,\beta}}$ is one-to-one.
On the other hand, we have $C_{\pi_{\lambda,\beta}}^{-1}u=\psi$ which
implies $C_{\pi_{\lambda,\beta}}$ is onto. Now, it follows that
\[
\|C_{\pi_{\lambda,\beta}}^{-1}u\|_{q,1,\pi_{\lambda,\beta}}\le\vertiii{u}_{q,1}^{2}(1-2^{\frac{2-q}{2}}\varepsilon^{-1})^{-2}(1-2^{q}G_{\varepsilon}^{2})^{-1}<\infty
\]
which proves the continuity of $C_{\pi_{\lambda,\beta}}^{-1}$.
\end{proof}

\subsection{Characterization of $(N)_{\pi_{\lambda,\beta}}^{-\kappa}$, $\kappa\in[0,1)$}

For every $l\in\mathbb{N}_{0}$, we denote by $\mathcal{E}_{2^{l}}^{k'}(\mathbb{C})$
the set of entire functions of growth $k'\in[2,\infty)$ and type
$2^{l}$, i.e., 
\[
\mathcal{E}_{2^{l}}^{k'}(\mathbb{C})=\big\{ U:\mathbb{C}\rightarrow\mathbb{C}~\text{is entire}\mid|U(z)|\leq C\exp(2^{l}|z|^{k'}),~C>0\big\}.
\]
For all natural $l$, the map 
\[
|\cdot|_{l,k'}:\mathcal{E}_{2^{l}}^{k'}(\mathbb{C})\rightarrow\mathbb{R},U\mapsto|U|_{l,k}:=\sup_{z\in\mathbb{C}}\big\{|U(z)|\exp(-2^{l}|z|^{k'})\big\}
\]
is a norm in $\mathcal{E}_{2^{l}}^{k'}(\mathbb{C})$ and $\big(\mathcal{E}_{2^{l}}^{k'}(\mathbb{C}),|\cdot|_{l,k'}\big)$
is a Banach space. The space of entire functions of maximal type and
order of growth $k'$, is defined by 
\[
\mathcal{E}_{\max}^{k'}(\mathbb{C}):=\bigcup_{l\in\mathbb{N}_{0}}\mathcal{E}_{2^{l}}^{k'}(\mathbb{C}).
\]
Taking into account Taylor's representation around the origin of $U\in\mathcal{E}_{\max}^{k'}(\mathbb{C})$,
we can consider in $\mathcal{E}_{\max}^{k'}(\mathbb{C})$ the family
$(\vertiii{\cdot}_{-q,-\kappa})_{q\in\mathbb{\mathbb{N}}_{0}},$ $\kappa\in[0,1)$,
of Hilbert norms, given by 
\[
\vertiii{U}_{-q,-\kappa}^{2}=\sum_{n=0}^{\infty}(n!)^{1-\beta}2^{-nq}|U_{n}|^{2}<\infty,
\]
which is equivalent to the family of norms $(|\cdot|_{l,k'})_{l\in\mathbb{N}_{0}}$,
see Appendix \ref{sec:Equivalent-norms}.
\begin{thm}
\label{thm:SqEmax}The $S_{\pi_{\lambda,\beta}}$-transform is a homeomorphism
between the generalized function space $(N)_{\pi_{\lambda,\beta}}^{-\kappa}$,
$\kappa\in[0,1)$ and the space $\mathcal{E}_{\max}^{k'}(\mathbb{C})$
of entire functions of maximal type and order of growth $k'=\frac{2}{1-\beta}$.
\end{thm}

\begin{proof}
The proof of this theorem is analogous to that of the previous theorem,
using Proposition \ref{prop:p18} in Appendix \ref{sec:Equivalent-norms}.
\end{proof}

\subsection{Characterization of $(N)_{\pi_{\lambda,\beta}}^{-1}$}

Let ${\rm Hol}_{0}(\mathbb{C})$ be the set of holomorphic functions
in a neighborhood of the origin. Let us consider the family $(|\cdot|_{l})_{l\in\mathbb{N}_{0}}$
of norms in ${\rm Hol}_{0}(\mathbb{C})$, defined as follows. If $U\in{\rm Hol}_{0}(\mathbb{C}),$
then there is $l\in\mathbb{N}_{0},$ such that 
\[
|U|_{l,\infty}:=\sup_{|z|\leq2^{-l}}|U(z)|<\infty.
\]
Under these conditions, $U$ admits a Taylor series representation
in $D=\{z\in\mathbb{C}~|~|z|\leq2^{-l}\},$ $U(z)=\sum_{n=0}^{\infty}U_{n}z^{n},$
$\forall z\in D,$ there exist $q\in\mathbb{N}_{0},$ such that 
\[
\vertiii{U}_{-q,-1}^{2}=\sum_{n=0}^{\infty}2^{-nq}|U_{n}|^{2}<\infty.
\]
Thus, we have introduce in ${\rm Hol}_{0}(\mathbb{C})$ another family
of norms, $(\vertiii{\cdot}_{-q,-1}^{2})_{q\in\mathbb{N}_{0}},$ equivalent
to the first ($|\cdot|_{l})_{l\in\mathbb{N}_{0}}$), see Proposition
\ref{prop:p19} in Appendix \ref{sec:Equivalent-norms}.
\begin{thm}
\label{thm:t13} The $S_{\pi_{\lambda,\beta}}$-transform is a homeomorphism
between the generalized function space $(N)_{\pi_{\lambda,\beta}}^{-1}$
and the ${\rm Hol}_{0}(\mathbb{C})$ space.
\end{thm}

\begin{proof}
The proof of this theorem is similar to that of the two previous theorems,
using Proposition \ref{prop:p19} in Appendix \ref{sec:Equivalent-norms}.
\end{proof}

\section{Wick Calculus}

\label{sec:Wick-Calculus}It is easy to see that the set ${\rm Hol}_{0}(\mathbb{C})$
forms an algebra of functions, for the usual operations of addition
and multiplication by a scalar. So, if $\Phi,\Psi\in(N)_{\pi_{\lambda,\beta}}^{-1},$
then $S_{\pi_{\lambda,\beta}}\Phi,S_{\pi_{\lambda,\beta}}\Psi\in{\rm Hol}_{0}(\mathbb{C})$
and also $S_{\pi_{\lambda,\beta}}\Phi\cdot S_{\pi_{\lambda,\beta}}\Psi\in{\rm Hol}_{0}(\mathbb{C})$.
By Theorem \ref{thm:t13}, there exists $\Theta\in(N)_{\pi_{\lambda,\beta}}^{-1}$
such that $S\Theta=S_{\pi_{\lambda,\beta}}\Phi\cdot S_{\pi_{\lambda,\beta}}\Psi.$
We denote the generalized function $\Theta$ by $\Phi\Diamond\Psi$
which we call \emph{Wick product} of $\Phi$ and $\Psi$. So if 
\[
\Phi=\sum_{n=0}^{\infty}\Phi_{n}Q_{n}^{\pi_{\lambda,\beta}},~~~~\Psi=\sum_{n=0}^{\infty}\Psi_{n}Q_{n}^{\pi_{\lambda,\beta}},
\]
then $\Phi\Diamond\Psi$ is represented by 
\[
\Phi\Diamond\Psi=\sum_{n=0}^{\infty}\Theta_{n}Q_{n}^{\pi_{\lambda,\beta}},~~~~\Theta_{n}=\sum_{k=0}^{n}\Phi_{k}\Psi_{n-k}.
\]
The following proposition tells us that the Wick product is a continuous
map.
\begin{prop}
\label{prop:wph-qh-p}The Wick product is a continuous mapping in
$(N)_{\pi_{\lambda,\beta}}^{-1}$, and given $\Phi\in(H)_{-q,\pi_{\lambda,\beta}}^{-1}$,
$\Psi\in(H)_{-p,\pi_{\lambda,\beta}}^{-1}$ we have the following:
\[
\|\Phi\Diamond\Psi\|_{-r,-1,\pi_{\lambda,\beta}}\leq\|\Phi\|_{-q,-1,\pi_{\lambda,\beta}}\|\Psi\|_{-p,-1,\pi_{\lambda,\beta}},\qquad r=p+q+1.
\]
\end{prop}

\begin{proof}
We can estimate the norm as follows: 
\[
\begin{array}{rcl}
\|\Phi\Diamond\Psi\|_{-r,-1,\pi_{\lambda,\beta}}^{2} & = & {\displaystyle \sum_{n=0}^{\infty}2^{-nr}|\Theta_{n}|^{2}}\\
 & \leq & {\displaystyle \sum_{n=0}^{\infty}2^{-nr}(n+1)\sum_{k=0}^{n}|\Phi_{k}|^{2}|\Psi_{n-k}|^{2}}\\
 & \leq & {\displaystyle \sum_{n=0}^{\infty}\sum_{k=0}^{n}2^{-nq}2^{-np}|\Phi_{k}|^{2}|\Psi_{n-k}|^{2}}\\
 & \leq & \left({\displaystyle \sum_{n=0}^{\infty}}2^{-nq}|\Phi_{n}|^{2}\right)\left({\displaystyle \sum_{n=0}^{\infty}}2^{-np}|\Psi_{n}|^{2}\right)\\
 & = & \|\Phi\|_{-q,-1,\pi_{\lambda,\beta}}^{2}\|\Psi\|_{-p,-1,\pi_{\lambda,\beta}}^{2},
\end{array}
\]
where we use the fact $\frac{n+1}{2^{n}}\leq1.$
\end{proof}
The Wick powers $\Phi^{\Diamond n}:=\Phi\Diamond\Phi\Diamond\dots\Diamond\Phi,$
$n\in\mathbb{N}_{0}$ of $\Phi\in(N)_{\pi_{\lambda,\beta}}^{-1}$
are defined as 
\[
\Phi^{\Diamond n}=S_{\pi_{\lambda,\beta}}^{-1}((S_{\pi_{\lambda,\beta}}\Phi)^{n}).
\]
In general, if $L(z)=\sum_{n=0}^{\infty}L_{n}z^{n}$ is an analytic
function, so we can study the Wick version of this function, i.e.,
$L^{\Diamond}(\Phi)$, $\Phi\in(N)_{\pi_{\lambda,\beta}}^{-1}.$ We
have the following theorem.
\begin{thm}
Let $L$ be an analytic function in a neighborhood of $z_{0}=\mathbb{E}_{\pi_{\lambda,\beta}}(\Phi),$
$\Phi\in(N)_{\pi_{\lambda,\beta}}^{-1}$. Then $L^{\Diamond}(\Phi)$
defined by $L^{\Diamond}(\Phi):=S_{\pi_{\lambda,\beta}}^{-1}(L(S_{\pi_{\lambda,\beta}}\Phi))$
is an element in $(N)_{\pi_{\lambda,\beta}}^{-1}$.
\end{thm}

\begin{proof}
In order to apply the characterization theorem (cf. Theorem \ref{thm:t13})
it is enough to check that $L(S_{\pi_{\lambda,\beta}}\Phi)$ is holomorphic
around the origin. But this follows easily by choosing a sufficiently
small neighborhood around the origin so that the composition $L\circ S_{\pi_{\lambda,\beta}}\Phi$
is holomorphic.
\end{proof}
\begin{example}
Let $\Phi,\Psi\in(N)_{\pi_{\lambda,\beta}}^{-1}$ be two generalized
functions.
\end{example}

\begin{enumerate}
\item Then $\exp^{\Diamond}(\Phi)$ can be written as 
\[
\exp^{\Diamond}(\Phi)=\sum_{n=0}^{\infty}\frac{1}{n!}\Phi^{\Diamond n}.
\]
\item It is easy to check the following property: 
\[
\exp^{\Diamond}(\Phi)\Diamond\exp^{\Diamond}(\Psi)=\exp^{\Diamond}(\Phi+\Psi).
\]
\item If $\Phi$ is such that $\mathbb{E}_{\pi_{\lambda,\beta}}(\Phi)>0$,
then $\log^{\Diamond}(\Phi)\in(N)_{\pi_{\lambda,\beta}}^{-1}$ which
is the solution of the equation 
\[
\exp^{\Diamond}(X)=\Phi.
\]
\item If $\Phi$ is such that $\mathbb{E}_{\pi_{\lambda,\beta}}(\Phi)\neq0$,
then $\Phi^{\Diamond-1}:=S_{\pi_{\lambda,\beta}}^{-1}((S_{\pi_{\lambda,\beta}}\Phi)^{-1})\in(N)_{\pi_{\lambda,\beta}}^{-1}$
and the solution of the equation 
\[
X\Diamond\Phi=\Psi
\]
is $X=\Phi^{\Diamond-1}\Diamond\Psi.$
\end{enumerate}
The Wick product easily generalizes to generalized function spaces
$(N)_{\pi_{\lambda,\beta}}^{-\kappa},$ $\kappa\in[0,1),$ as well
as test function spaces $(N)_{\pi_{\lambda,\beta}}^{\kappa}$, $\kappa\in[0,1]$.
The concept of Wick product is often used in models of stochastic
differential equations, the solutions being obtained through the $S_{\pi_{\lambda,\beta}}$--transform,
see for example \cite{HOUZ09}.

\section{Conclusion and Outlook}

In this paper, we have developed the biorthogonal system to investigate
the spaces of test and generalized functions associated to the fractional
Poisson measure $\pi_{\lambda,\beta}$ on $\mathbb{N}_{0}$. The system
of generalized Appell polynomials $\mathbb{P}^{\pi_{\lambda,\beta}}$
describes the spaces of test functions while the generalized dual
Appell system $\mathbb{Q}^{\pi_{\lambda,\beta}}$ is suited to description
of generalized functions rising from $\pi_{\lambda,\beta}$. In addition,
we have characterized both test and generalized function spaces through
two suitable integral transforms. The Wick calculus extends in a straightforward
manner from Gaussian analysis to the present fractional Poisson analysis.
In a future work, we plan to extend these results to an infinite dimensional
framework. More precisely, the fractional Poisson measure on the linear
space $\mathcal{D}'(\mathbb{R}^{d})=C_{0}^{\infty}(\mathbb{R}^{d})$,
the dual of the nuclear space of Schwartz test functions, or on the
configuration space $\Gamma_{\mathbb{R}^{d}}\subset\mathcal{D}'(\mathbb{R}^{d})$
over $\mathbb{R}^{d}$.

\appendix

\section{Proofs. Equivalence of Norms}

\label{sec:Equivalent-norms}For the completeness of this work we
provide the proofs of the following propositions.
\begin{prop}
\label{prop:equivalence-norms} The two system of norms $(|\cdot|_{l,k})_{l\in\mathbb{N}_{0}}$,
$k=\frac{2}{1+\kappa}$ and $(\vertiii{\cdot}_{q,\kappa})_{q\in\mathbb{N}_{0}},$
$\kappa\in[0,1]$ defined in the space $\mathcal{E}_{\min}^{k}(\mathbb{C})$
are equivalent.
\end{prop}

\begin{proof}
Let $l\in\mathbb{N}_{0}$ and $u\in\mathcal{E}_{\min}^{k}(\mathbb{C})$
be arbitrary. We are going to show that there exist $C>0$ and $q\in N_{0},$
such that $|u|_{l,k}\leq C\vertiii{u}_{q,\kappa}.$ As $u\in\mathcal{E}_{\min}^{k}(\mathbb{C}),$
we have that $\vertiii{u}_{q,\kappa}^{2}<\infty,$ $\forall q\in\mathbb{N}_{0}$.
In addition, it turns out also that $|u|_{l,k}<\infty,$ $\forall l\in\mathbb{N}_{0}$
because $\mathcal{E}_{2^{-l}}^{k}(\mathbb{C})\subset\mathcal{E}_{\min}^{k}(\mathbb{C})$,
$\forall l\in\mathbb{N}_{0}.$ Since $u$ is entire, it admits the
Taylor series representation around the origin, $u(z)=\sum_{n=0}^{\infty}u_{n}z^{n}.$
Thus, $\forall q\in\mathbb{N}_{0},$ $\forall z\in\mathbb{C},$ by
the Cauchy-Schwarz inequality we obtain 
\[
\begin{array}{ccl}
|u(z)| & \leq & {\displaystyle \sum_{n=0}^{\infty}|u_{n}||z|^{n}}\\
 & \leq & {\displaystyle \sum_{n=0}^{\infty}\left((n!)^{1+\kappa}2^{nq}\right)^{1/2}|u_{n}|\left((n!)^{1+\kappa}2^{nq}\right)^{-1/2}|z|^{n}}\\
 & \leq & \vertiii{u}_{q,\kappa}{\displaystyle \sqrt{{\displaystyle \sum_{n=0}^{\infty}\frac{(2^{-q}|z|^{2})^{n}}{(n!)^{1+\kappa}}}}}\\
 & = & \vertiii{u}_{q,\kappa}{\displaystyle \sqrt{{\displaystyle \sum_{n=0}^{\infty}\left(\frac{(2^{-kq/2}|z|^{k})^{n}}{n!}\right)^{1+\kappa}}}}.
\end{array}
\]
Note that by the H{\"o}lder's inequality, we have
\begin{align*}
{\displaystyle \sum_{n=0}^{\infty}\left(\frac{x^{n}}{n!}\right)^{\gamma}} & \leq\left(\sum_{n=0}^{\infty}\frac{x^{n}}{n!}\right)^{2-\gamma}\left(\sum_{n=0}^{\infty}\left(\frac{x^{n}}{n!}\right)^{2}\right)^{\gamma-1}\\
 & \leq(\mathrm{e}^{x})^{2-\gamma}(\mathrm{e}^{2x})^{\gamma-1}=\mathrm{e}^{\gamma x}
\end{align*}
for $x\geq0$ and $1<\gamma<2$. This yields
\[
\sqrt{{\displaystyle \sum_{n=0}^{\infty}\left(\frac{(2^{-kq/2}|z|^{k})^{n}}{n!}\right)^{1+\kappa}}}\leq\exp(k^{-1}2^{-kq/2}|z|^{k}).
\]
Thus, we have

\[
|u(z)|\leq\vertiii{u}_{q,\kappa}\exp(2^{-kq/2}|z|^{k}).
\]
So, $\forall q\in\mathbb{N}_{0,}$ 
\[
{\displaystyle \sup_{z\in\mathbb{C}}\{|u(z)|\exp(-2^{-kq/2}|z|^{k})\}\leq\|u\|_{q,\kappa.}}
\]
Hence, we choose $q$ such that $l\leq\frac{q}{1+\kappa}$ to get
$|u|_{l,k}\leq\|u\|_{q,\kappa.}$

Conversely, let $q\in\mathbb{N}_{0}$ and $u\in\mathcal{E}_{\min}^{k}(\mathbb{C})$
be arbitrary. Now we are going to show that there exist $K>0$ and
$l\in\mathbb{N}_{0}$, such that $\|u\|_{q,\kappa}\leq K|u|_{l,k}.$
Note that, $\forall l\in\mathbb{N}_{0},$ we have 
\[
|u(z)|\leq|u|_{l,k}\exp(2^{-l}|z|^{k}),\ \forall z\in\mathbb{C}.
\]
Thus, it turns out that 
\[
|u_{n}|\leq{\displaystyle \frac{1}{2\pi}\int_{|z|=\rho>0}\frac{|u(z)|}{|z|^{n+1}}|dz|}\leq|u|_{l,k}{\displaystyle \frac{\mathrm{e}^{2^{-l}\rho^{k}}}{\rho^{n}}},\ \rho>0.
\]
This bound may be optimized by taking $\rho=(n2^{l}k^{-1})^{k^{-1}}.$
Substituting, we get 
\[
|u_{n}|\leq|u|_{l,k}(\mathrm{e}^{n}n^{-n})^{k^{-1}}(k^{n}2^{-nl})^{k^{-1}}.
\]
Finally, taking into the account the Stirling formula, 
\[
\mathrm{e}^{n}n^{-n}\leq{\displaystyle \frac{1}{n!}\mathrm{e}\sqrt{2\pi}(n+1)}
\]
we obtain 
\[
|u_{n}|\leq|u|_{l,k}\left(\frac{\mathrm{e}\sqrt{2\pi}(n+1)}{n!}\right)^{k^{-1}}(k^{n}2^{-nl})^{k^{-1}}.
\]
Thus, the norm $\|u\|_{q,\kappa}^{2}$ is given by

\begin{equation}
\vertiii{u}_{q,\kappa}^{2}=\sum_{n=0}^{\infty}(n!)^{1+\kappa}2^{nq}|u_{n}|^{2}\leq|u|_{l,k}^{2}(\mathrm{e}\sqrt{2\pi})^{1+\kappa}\sum_{n=0}^{\infty}(n+1)^{2}(2^{q}(k2^{-l})^{1+\kappa})^{n}.\label{eq:eq9}
\end{equation}
Given $q\in\mathbb{N}_{0},$ we take $l\in\mathbb{N}_{0}$ such that
$2^{q}(k2^{-l})^{1+\kappa}<1$ and use the equality

\[
{\displaystyle \sum_{n=0}^{\infty}(n+1)^{2}x^{n}}\leq\sum_{n=0}^{\infty}(n+1)(n+2)x^{n}=(1-x)^{-3},\ |x|<1,
\]
the series in \eqref{eq:eq9} is finite. The inequality \eqref{eq:eq9}
also proves that the norms of the family $(\vertiii{\cdot}_{q,\kappa})_{q\in\mathbb{\mathbb{N}}_{0}}$
are also well defined.
\end{proof}
\begin{prop}
\label{prop:p18}The two system of norms $(|\cdot|_{l,k'})_{l\in\mathbb{N}_{0}}$,
$k'=\frac{2}{1-\kappa}$ and $(\vertiii{\cdot}_{-q,-\kappa})_{q\in\mathbb{N}_{0}},$
$\kappa\in[0,1)$ defined in the space $\mathcal{E}_{\max}^{k'}(\mathbb{C})$
are equivalent.
\end{prop}

\begin{proof}
Following a process analogous to the previous proof, taking $q\in\mathbb{N}_{0}$
and $U\in\mathcal{E}_{\max}^{k'}(\mathbb{C})$ as arbitrary, we show
that there is a natural number $l$, such that $2^{l}\geq(k')^{-1}2^{(q+\kappa)/(1-\kappa)}$
and a positive number $C=2^{\kappa/2}$ such that $|u|_{l,k'}\leq C\vertiii{U}_{-q,-\kappa}$.

Conversely, and also analogously to the previous proof, given $l\in\mathbb{N}_{0}$
and $U\in\mathcal{E}_{\max}^{k'}(\mathbb{C})$ such that $|U|_{l,k'}<\infty$,
we prove that there is a natural number $q$ where $2^{-q}(\frac{2^{l+1}}{1-\kappa}){}^{1-\kappa}<1$
and a positive number $K=(\mathrm{e\sqrt{2\pi}})^{(1-\kappa)/2}(1-2^{-q}(\frac{2^{l+1}}{1-\kappa}){}^{\kappa-1}),$
such that $\vertiii{U}_{-q,-\kappa}\leq K|U|_{l,k'}$.
\end{proof}
\begin{prop}
\label{prop:p19}The two system of norms $(|\cdot|_{l,\infty})_{l\in\mathbb{N}_{0}}$
and $(\vertiii{\cdot}_{-q,-1})_{q\in\mathbb{N}_{0}},$ introduced
in ${\rm Hol}_{0}(\mathbb{C})$ are equivalent.
\end{prop}

\begin{proof}
The proof is analogous to the two previous ones. Taking $q\in\mathbb{N}_{0}$
and $U\in\mathrm{Hol_{0}(\mathbb{C})}$, so that $\vertiii{U}_{-q,-1}<\infty$,
we show that there exist $C=(1-2^{q}|z|^{2})^{-2^{-1}}>0$ and $l\in\mathbb{N}_{0}$
$(l>\frac{q}{2})$, such that $|U|_{l,\infty}\leq C\vertiii{U}_{-q,-1}.$
In turn, given $l\in\mathbb{N}_{0}$ and $U\in\mathrm{Hol_{0}(\mathbb{C})}$,
with $|U|_{l,\infty}<\infty$, we obtain $q\in\mathbb{N}_{0}$ $(q>2l)$,
and $K=(1-2^{-q+2l})^{-2^{-1}}>0$ with $\vertiii{U}_{-q,-1}\leq K|U|_{l,\infty.}$
\end{proof}

\section{Appell Polynomials and Bell Polynomials}

\label{sec:Appell-Bell-connection}In this section, we recall some
fundamental concepts and results well known in the literature and
needed in what follows. More precisely, the notion of Stirling numbers
and Bell polynomials as well as Appell sequences are introduced (see
Subsection \ref{subsec:Stirling-Numbers}) and their connection with
Appell polynomials are developed in Subsection \ref{subsec:Connection-Appell-Bell}.

\subsection{Stirling Numbers and Bell Polynomials}

\label{subsec:Stirling-Numbers}
\begin{defn}[Stirling numbers of the first kind, see \cite{Olver2010}]
\label{def:Stirling-numbers-first-kind} Given $n,m\in\mathbb{N}_{0}$,
with $0\leq m\leq n$, the \textit{Stirling numbers of the first kind},
denoted by $s(n,m)$, is defined as $(-1)^{n-m}$ times the number
of ways to arrange $n$ distinct objects around $m$ (indistinguishable)
circles such that each circle has at least one object.
\end{defn}

\begin{defn}[Stirling numbers of the second kind, see \cite{CM92}]
\label{def:Stirling-numbers-second-kind} Given two nonnegative integers
$n$ and $m$, the \textit{Stirling numbers of the second kind}, denoted
by $S(n,m)$, is defined as the number of ways of distributing $n$
distinct objects into $m$ identical boxes such that no box is empty.
\end{defn}

\begin{prop}[see \cite{Olver2010}]
\label{prop:Stirling-generating-functions}The Stirling numbers have
the following generating functions \textup{
\[
(\mathrm{e}^{z}-1)^{k}=k!\sum_{n=k}^{\infty}S(n,k)\frac{z^{n}}{n!}\quad\mathrm{and}\quad(\log(1+z))^{k}=k!\sum_{n=k}^{\infty}s(n,k)\frac{z^{n}}{n!}.
\]
}
\end{prop}

\begin{prop}
\label{prop:Stirling-numbers-explicit} The Stirling numbers can be
expressed as
\end{prop}

\begin{enumerate}
\item $s(n,m)=\frac{A_{n}^{m}}{m!}$, where $A_{0}^{0}:=1$ and $A_{n}^{m}:={\displaystyle \!\!\!\!\!\!\sum_{l_{1}+\dots+l_{m}=n}\!\!\!\!\frac{(-1)^{n+m}n!}{l_{1}\dots l_{m}}}$,
$l_{i}\in\left\{ 1,\dots,n\right\} $ for all $i=1,\dots,m.$
\item $S(n,m)=\frac{B_{n}^{m}}{m!}$, where $B_{0}^{0}:=1$ and $B_{n}^{m}:=\!\!\!\!{\displaystyle \sum_{l_{1}+\dots+l_{m}=n}\frac{n!}{l_{1}!\dots l_{m}!}}$,
$l_{i}\in\left\{ 1,\dots,n\right\} $ for all $i=1,\dots,m.$
\end{enumerate}
\begin{proof}
Let $k,n,m\in\mathbb{N}_{0}$ be given. 
\begin{description}
\item [{1.\quad{}}] First, note that $s(0,0)=1$. By Proposition~\ref{prop:Stirling-generating-functions},
\begin{align*}
\sum_{n=m}^{\infty}s(n,m)\frac{z^{n}}{n!} & =\frac{(\log(1+z))^{m}}{m!}\\
 & =\frac{1}{m!}\left(\sum_{l=1}^{\infty}\frac{z^{l}}{l!}(-1)^{l+1}(l-1)!\right)^{m}\\
 & =\frac{1}{m!}\sum_{n=m}^{\infty}\frac{z^{n}}{n!}\sum_{l_{1}+\dots+l_{m}=n}\frac{n!}{l_{1}!\dots l_{m}!}\prod_{i=1}^{m}(-1)^{l_{i}+1}(l_{i}-1)!\\
 & =\frac{1}{m!}\sum_{n=m}^{\infty}\frac{z^{n}}{n!}\underbrace{\sum_{l_{1}+\dots+l_{m}=n}\frac{(-1)^{n+m}n!}{l_{1}\dots l_{m}}}_{=:A_{n}^{m}}\\
 & =\sum_{n=m}^{\infty}\frac{A_{n}^{m}}{m!}\frac{z^{n}}{n!}.
\end{align*}
 The result follows by comparing the coefficients in both sides of
the equation.
\item [{2.\quad{}}] Note that $S(0,0)=1$. By Proposition~\ref{prop:Stirling-generating-functions},
\begin{align*}
\sum_{n=m}^{\infty}S(n,m)\frac{z^{n}}{n!} & =\frac{(\mathrm{e}^{z}-1)^{m}}{m!}\\
 & =\frac{1}{m!}\left(\sum_{l=1}^{\infty}\frac{z^{l}}{l!}\right)^{m}\\
 & =\frac{1}{m!}\sum_{n=m}^{\infty}\frac{z^{n}}{n!}\underbrace{\sum_{l_{1}+\dots+l_{m}=n}\frac{n!}{l_{1}!\dots l_{m}!}}_{=:B_{n}^{m}}\\
 & =\sum_{n=m}^{\infty}\frac{B_{n}^{m}}{m!}\frac{z^{n}}{n!}.
\end{align*}
The claim follows again by comparing the coefficients in both sides
of the equation. \hfill{}$\qedhere$
\end{description}
\end{proof}

\begin{prop}[see \cite{Olver2010}]
\label{prop:Stirling-falling-factorial} For $n,k\in\mathbb{N}_{0}$,
we have
\end{prop}

\begin{enumerate}
\item $(x)_{n}=\sum_{k=0}^{n}s(n,k)x^{k}$,
\item $x^{n}=\sum_{k=0}^{n}S(n,k)(x)_{k}.$
\end{enumerate}
\begin{defn}[Appell sequences, see \cite{Appell82,YY09}]
\label{def:ap} A polynomial sequence $A_{n}(x)$ is said to be \textit{Appell
sequence for $g(t)$} if and only if 
\begin{equation}
\frac{1}{g(t)}\exp(xt)=\sum_{n=0}^{\infty}A_{n}(x)\frac{t^{n}}{n!},
\end{equation}
where 
\begin{equation}
g(t)=\sum_{n=0}^{\infty}g_{n}\frac{t^{n}}{n!}
\end{equation}
and $g_{0}\neq0.$
\end{defn}

\begin{thm}[General Leibniz Rule, see \cite{CM92}]
\label{thm:glr} If $f$ and $g$ are $n$-times differentiable functions,
then the product is also $n$-times differentiable and its $n$-th
derivative is given by 
\begin{equation}
(fg)^{(n)}=\sum_{k=0}^{n}\binom{n}{k}f^{(n-k)}g^{(k)},\label{eq:Leibniz-rule}
\end{equation}
where $\binom{n}{k}=\frac{n!}{k!(n-k)!}$ is the binomial coefficient
and $f^{(0)}=f.$
\end{thm}

\begin{thm}[Fa{\'a} di Bruno's Formula, see \cite{B56,J02}]
\emph{ \label{thm:fdb} }If $f$ and $g$ are functions for which
all necessary derivatives are defined, then 
\[
\frac{d^{n}}{dt^{n}}f(g(t))=\sum\frac{n!}{m_{1}!m_{2}!\dots m_{k}!}f^{(m_{1}+\dots+m_{k})}(g(t))\prod_{j=1}^{n}\left(\frac{g^{(j)}(t)}{j!}\right)^{m_{j}},
\]
where the sum is over all $n$-tuples of non-negative integers $(m_{1},\dots,m_{n})$
satisfying the constraint 
\[
1\cdot m_{1}+2\cdot m_{2}+\dots+n\cdot m_{n}=n.
\]
\end{thm}

\begin{defn}[\textit{Bell polynomials, see} \cite{CM92}]
\label{def:Bell-polynomials}The \textit{Bell polynomials} $B_{n,k}(\cdot)$,
$n,k\in\mathbb{N}_{0}$ are defined for any $x_{1},x_{2},\dots,x_{n-k+1}\in\mathbb{R}$
by
\[
B_{n,k}(x_{1},x_{2},\dots,x_{n-k+1}):=\sum\frac{n!}{j_{1}!j_{2}!\dots j_{n-k+1}!}\left(\frac{x_{1}}{1!}\right)^{j_{1}}\left(\frac{x_{2}}{2!}\right)^{j_{2}}\dots\left(\frac{x_{n-k+1}}{(n-k+1)!}\right)^{j_{n-k+1}},
\]
where the sum is taken over all sequences $j_{1},j_{2},\dots,j_{n-k+1}$
of non-negative integers such that the following conditions are satisfied
\begin{enumerate}
\item $j_{1}+j_{2}+\dots+j_{n-k+1}=k,$
\item $j_{1}+2j_{2}+\dots+(n-k+1)j_{n-k+1}=n.$
\end{enumerate}
\end{defn}

For convenience, we use the short notation 
\[
B_{n,k}((x_{j})_{j=1}^{n-k+1}):=B_{n,k}(x_{1},x_{2},\dots,x_{n-k+1}).
\]

\begin{thm}[Fa{\'a} di Bruno's formula using Bell polynomial, see \cite{B56,J02}]
\emph{ }\label{thm:Faa-Bruno-formula}If $f(t)$ and $g(t)$ are
functions for which all necessary derivatives are defined, then 
\[
\frac{d^{n}}{dt^{n}}f(g(t))=\sum_{k=0}^{n}f^{(k)}(g(t))B_{n,k}(g'(t),g''(t),\dots,g^{(n-k+1)}(t)).
\]
\end{thm}

\subsection{Connections Between Appell and Bell Polynomials}

\label{subsec:Connection-Appell-Bell} In general, the Appell polynomials
$A_{n}^{\lambda,\beta}(\cdot)$, $n\in\mathbb{N}$, may be explicitly
written in terms of the Bell polynomials and the moments of the fractional
Poisson measure $\pi_{\lambda,\beta}$, see Theorem \ref{thm:gen-Appell-Bell-relation}
below.

At first recall the definition of the moments of $\pi_{\lambda,\beta}$,
for any $n\in\mathbb{N}_{0}$, 
\[
M_{\lambda,\beta}(n):=\frac{\mathrm{d}^{n}}{\mathrm{d}z^{n}}l_{\pi_{\lambda,\beta}}(z)\big|_{z=0}=\frac{\mathrm{d}^{n}}{\mathrm{d}z^{n}}E_{\beta}(\lambda(\mathrm{e}^{z}-1))\big|_{z=0}.
\]

\begin{lem}
\label{lem:Appell-at-zero}For every $n\in\mathbb{N}_{0}$, we have
\[
A_{n}^{\lambda,\beta}(0)=\sum_{k=0}^{n}(-1)^{k}k!B_{n,k}\big((M_{\lambda,\beta}(j))_{j=1}^{n-k+1}\big).
\]
\end{lem}

\begin{proof}
By definition of $A_{n}^{\lambda,\beta}(0)$ we have
\[
A_{n}^{\lambda,\beta}(0):=\frac{\mathrm{d}^{n}}{\mathrm{d}z^{n}}\mathrm{e}_{\pi_{\lambda,\beta}}(z;x)\big|_{z=x=0}.
\]
We denote by $f$ and $g$ the functions
\begin{align*}
f(z) & :=\mathrm{e}_{\pi_{\lambda,\beta}}(z;0)=\frac{1}{l_{\pi_{\lambda,\beta}}(z)}=\frac{1}{E_{\beta}(\lambda(\mathrm{e}^{z}-1))},\\
g(z) & :=\mathrm{\mathrm{e}}^{zx}.
\end{align*}
Hence, using the general Leibniz rule \eqref{eq:Leibniz-rule}, $A_{n}^{\lambda,\beta}(0)$
may be expression as
\[
A_{n}^{\lambda,\beta}(0)=\frac{\mathrm{d}^{n}}{\mathrm{d}z^{n}}\big(f(z)g(z)\big)\big|_{z=x=0}=\sum_{k=0}^{n}\binom{n}{k}f^{(n-k)}(z)g^{(k)}(z)\big|_{z=x=0}.
\]
On one hand, it is clear that 
\[
g^{(k)}(z)=x^{k}g(z).
\]
On the other hand, to compute $f^{(n-k)}(z)$ we use the Fa{\'a}
di Bruno formula given in terms of Bell polynomials, see Theorem~\ref{thm:Faa-Bruno-formula}.
Namely, we represent 
\[
f(z)=h\big(l_{\pi_{\lambda,\beta}}(z)\big),
\]
where $h(z)=\frac{1}{z}$ and obtain
\begin{align*}
f^{(n-k)}(z) & =\sum_{j=0}^{n-k}h^{(j)}\big(l_{\pi_{\lambda,\beta}}(z)\big)B_{n-k,j}\big(l'_{\pi_{\lambda,\beta}}(z),l''_{\pi_{\lambda,\beta}}(z),\dots,l_{\pi_{\lambda,\beta}}^{(n-k-j+1)}(z)\big).\\
 & =\sum_{j=0}^{n-k}\frac{(-1)^{j}j!}{(l_{\pi_{\lambda,\beta}}(z))^{j+1}}B_{n-k,j}\big((l_{\pi_{\lambda,\beta}}^{(i)}(z))_{i=1}^{n-k-j+1}\big).
\end{align*}
Putting together yields
\[
\frac{\mathrm{d}^{n}}{\mathrm{d}z^{n}}\big(f(z)g(z)\big)=\sum_{k=0}^{n}\binom{n}{k}\left[\sum_{j=0}^{n-k}\frac{(-1)^{j}j!}{(l_{\pi_{\lambda,\beta}}(z))^{j+1}}B_{n-k,j}\big((l_{\pi_{\lambda,\beta}}^{(i)}(z))_{i=1}^{n-k-j+1}\big)\right]x^{k}g(z)
\]
and evaluating this expression at $z=x=0$ gives the claimed result.
\end{proof}
\begin{thm}
\label{thm:Appell-Bell-relation}The Appell polynomials $A_{n}^{\lambda,\beta}(\cdot)$,
$n\in\mathbb{N},$ generated by the fractional Poisson measure $\pi_{\lambda,\beta}$
are given for any $x\in\mathbb{R}$ in terms of the Bells polynomials
as
\begin{equation}
A_{n}^{\lambda,\beta}(x)=\sum_{k=0}^{n}\binom{n}{k}\Bigg[\sum_{i=0}^{n-k}(-1)^{i}i!B_{n-k,i}((M_{\lambda,\beta}(j))_{j=1}^{n-k-i+1})\Bigg]x^{k}.\label{eq:bell-Apell}
\end{equation}
\end{thm}

\begin{proof}
It follows from \eqref{eq:Appell-z^n}, Lemma \ref{lem:Appell-at-zero}
and the uniqueness of the representation in terms of the usual powers
$x^{k}$.
\end{proof}
\begin{rem}
The property \eqref{eq:bell-gen-appell} of the Appell polynomials
$A_{n}^{\lambda,\beta}$, $n\in\mathbb{N}_{0}$, is not specific of
the fractional Poisson measure $\pi_{\lambda,\beta}$. In fact, given
a probability space $(\mathbb{R},\mathcal{B}(\mathbb{R}),\mu)$ such
that the Laplace transform $l_{\mu}$ is an analytic function, then
the Appell polynomial sequence generated by the corresponding Wick
exponential $\mathrm{e}_{\mu}(\cdot;x)$ also satisfies the mentioned
property. This follows from the definition of the Wick exponential
and the same arguments as in the proof of Lemma \ref{lem:gen-Appell-at-zero}
and again Proposition~\ref{prop:gen-Appell-prop}--(P2) from Proposition
\ref{prop:gen-Appell-prop}.
\end{rem}

\subsection{Connections Between Generalized Appell Polynomials and Bell Polynomials}

\label{subsec:Connection-Gen-Appell-Bell} We also give an explicit
representation of the generalized Appell polynomials $C_{n}^{\lambda,\beta}(\cdot)$,
$n\in\mathbb{N}$, in terms of the Bell polynomials, see Theorem \ref{thm:gen-Appell-Bell-relation}
below.
\begin{lem}
\label{lem:gen-Appell-at-zero}For every $n\in\mathbb{N}_{0}$, we
have
\[
C_{n}^{\lambda,\beta}(0)=\sum_{j=0}^{n}(-1)^{j}j!B_{n,j}\left(\left(\tilde{M}_{\lambda,\beta}(m)\right)_{m=1}^{n-j+1}\right).
\]
\end{lem}

\begin{proof}
Following a process analogous to the proof of Lemma \ref{lem:Appell-at-zero},
we use the general Leibniz rule on the polynomials $C_{n}^{\lambda,\beta}(\cdot)$,
$n\in\mathbb{N}$ i.e.,
\[
C_{n}^{\lambda,\beta}(0)=\frac{\mathrm{d}^{n}}{\mathrm{d}z^{n}}\big(f(z)g(z)\big)\big|_{z=x=0}=\sum_{k=0}^{n}\binom{n}{k}f^{(n-k)}(z)g^{(k)}(z)\big|_{z=x=0}
\]
where 
\begin{align*}
f(z) & :=\mathrm{e}_{\pi_{\lambda,\beta}}(z;0)=\frac{1}{l_{\pi_{\lambda,\beta}}(\log(1+z))}=\frac{1}{E_{\beta}(\lambda z)},\\
g(z) & :=\mathrm{\mathrm{e}}^{x\log(1+z)}.
\end{align*}
Here, we use the Fa{\'a} di Bruno formula to obtain
\begin{align*}
f^{(n-k)}(z) & =\sum_{j=0}^{n-k}\frac{(-1)^{j}j!}{(l_{\pi_{\lambda,\beta}}(z))^{j+1}}B_{n-k,j}\left(\left(\sum_{p=i}^{\infty}\frac{p!}{(p-i)!}\frac{\lambda^{p}z^{p-i}}{\Gamma(p\beta+1)}\right)_{i=1}^{n-k-j+1}\right)
\end{align*}
and 
\begin{align*}
g^{(k)}(z) & =\sum_{l=0}^{k}x^{l}\mathrm{e}^{x\log(1+z)}B_{k,l}\left(\left(\frac{(-1)^{n+1}(n-1)!}{(1+z)^{n}}\right)_{m=1}^{k-l+1}\right).\qedhere
\end{align*}
\end{proof}
\begin{thm}
\label{thm:gen-Appell-Bell-relation}The generalized Appell polynomials
$C_{n}^{\lambda,\beta}(\cdot)$, $n\in\mathbb{N},$ generated by the
fractional Poisson measure $\pi_{\lambda,\beta}$ are given for any
$x\in\mathbb{R}$ in terms of the Bells polynomials as
\begin{equation}
C_{n}^{\lambda,\beta}(x)=\sum_{k=0}^{n}\binom{n}{k}\Bigg[\sum_{i=0}^{n-k}(-1)^{i}i!B_{n-k,i}\left(\left(\tilde{M}_{\lambda,\beta}(j)\right)_{j=1}^{n-k-i+1}\right)\Bigg](x)_{k},\label{eq:bell-gen-appell}
\end{equation}
where $(x)_{k}$ are the falling factorials.
\end{thm}

\begin{proof}
The proof follows from Proposition~\ref{prop:gen-Appell-prop}--(P4)
and Lemma \ref{lem:gen-Appell-at-zero}.
\end{proof}

\section*{Acknowledgments}

This work was partially supported by the Center for Research in Mathematics
and Applications (CIMA-UMa) related with the Statistics, Stochastic
Processes and Applications (SSPA) group, through the grant UIDB/MAT/04674/2020
of FCT-Funda{\c c\~a}o para a Ci{\^e}ncia e a Tecnologia, Portugal
and also by the Complex systems group of the Premiere Research Institute
of Science and Mathematics (PRISM), MSU-Iligan Institute of Technology.
Financial support of the Department of Science and Technology --
Accelerated Science and Technology Human Resource Development Program
(DOST-ASTHRDP) of the Philippines under the Research Enrichment (Sandwich)
Program is gratefully acknowledge.

\end{document}